\documentclass[a4paper,12pt]{article}
\usepackage[english]{babel}
\usepackage{amssymb}
\usepackage{amsmath}
\usepackage{amsthm}
\usepackage[utf8]{inputenc}
\usepackage[T1]{fontenc}
\usepackage[all]{xy}
\usepackage{extarrows}
\usepackage[usenames,dvipsnames]{color}
\usepackage{endnotes}
\usepackage{cancel}
\usepackage{paralist}
\usepackage{mathtools}
\usepackage{setspace}
\usepackage[pdftex, pdfusetitle,colorlinks=false,pdfborder={0 0 0}]{hyperref}

\usepackage{eucal}
\newcommand{\macrocolor}{}

\newtheorem{theorem}{Theorem}
\newtheorem{corollary}[theorem]{Corollary}
\newtheorem{proposition}[theorem]{Proposition}

\newtheorem{lemma}[theorem]{Lemma}
\theoremstyle{definition}
\newtheorem{definition}[theorem]{Definition}
\theoremstyle{remark}
\newtheorem{remark}[theorem]{Remark}
\newtheorem{example}[theorem]{Example}

\newcommand{\StSKloc}{\widetilde{\textrm{SK}}_{\textrm{loc}}}
\newcommand{\SbasElocq}{\widetilde{\mathcal{E}}_{\mathrm{loc}}}
\newcommand{\basElocq}[1]{\SbasElocq(#1)}

\newcommand{\D}{{\macrocolor\mathrm{D}}}
\newcommand{\SK}[1]{{\macrocolor\mathrm{SK}(#1)}}

\newcommand{\SKnets}[1]{{\macrocolor\widetilde{\mathrm{SK}}(#1)}}
\newcommand{\SKloc}[1]{{\macrocolor\mathrm{SK}_{\mathrm{loc}}(#1)}}
\newcommand{\SKlocnets}[1]{{\macrocolor\widetilde{\mathrm{SK}}_{\mathrm{loc}}(#1)}}
\newcommand{\SSKlocnets}{{\macrocolor\widetilde{\mathrm{SK}}_{\mathrm{loc}}}}
\newcommand{\Cinfnets}[1]{{\macrocolor \widetilde{C^\infty}(#1)}}

\newcommand{\basE}[1]{{\macrocolor \mathcal{E}(#1)}}

\newcommand{\basEloc}[1]{{\macrocolor \mathcal{E}_{loc}(#1)}}

\newcommand{\basEploc}[1]{{\macrocolor \mathcal{E}_{ploc}(#1)}}
\newcommand{\basEpi}[1]{{\macrocolor \mathcal{E}_{pi}(#1)}}
\newcommand{\basL}[1]{{\macrocolor \mathcal{L}(#1)}}
\newcommand{\basLloc}[1]{{\macrocolor \mathcal{L}_{loc}(#1)}}
\newcommand{\basLploc}[1]{{\macrocolor \mathcal{L}_{ploc}(#1)}}
\newcommand{\basLpi}[1]{{\macrocolor \mathcal{L}_{pi}(#1)}}

\newcommand{\basLcinf}[1]{{\macrocolor \mathcal{L}_{C^\infty}(#1)}}

\newcommand{\SbasEloc}{{\macrocolor \mathcal{E}_{loc}}}

\newcommand{\SbasEploc}{{\macrocolor \mathcal{E}_{ploc}}}
\newcommand{\SbasEpi}{{\macrocolor \mathcal{E}_{pi}}}

\newcommand{\SbasLloc}{{\macrocolor \mathcal{L}_{loc}}}
\newcommand{\SbasLploc}{{\macrocolor \mathcal{L}_{ploc}}}

\newcommand{\Lsm}{\Lie}
\newcommand{\Lnf}{\Lie}
\newcommand{\Lsk}{\Lie^{\textrm{SK}}}
\newcommand{\LE}{\widehat{\Lie}}
\newcommand{\tLE}{\widetilde\Lie}

\providecommand{\snorm}[1]{\lVert#1\rVert}

\DeclareMathOperator{\Div}{div}
\DeclareMathOperator{\cs}{cs}

\newcommand{\ud}{\mathrm{d}}
\newcommand{\pd}{\partial}
\newcommand{\e}{\varepsilon}

\newcommand{\bC}{\mathbb{C}}

\newcommand{\bN}{\mathbb{N}}
\newcommand{\bR}{\mathbb{R}}

\newcommand{\R}{\bR}

\newcommand{\cA}{\mathcal{A}}

\newcommand{\cD}{\mathcal{D}}
\newcommand{\cE}{\mathcal{E}}

\newcommand{\cG}{\mathcal{G}}

\newcommand{\cL}{\mathcal{L}}
\newcommand{\cM}{\mathcal{M}}
\newcommand{\cN}{\mathcal{N}}
\newcommand{\cO}{\mathcal{O}}

\newcommand{\cS}{\mathcal{S}}
\newcommand{\cT}{\mathcal{T}}

\newcommand{\coleq}{\mathrel{\mathop:}=}

\DeclareMathOperator{\Fl}{Fl}

\DeclareMathOperator{\ev}{ev}

\DeclareMathOperator{\id}{id}

\DeclareMathOperator{\supp}{supp}

\newcommand{\Lie}{\mathrm{L}}

\providecommand{\norm}[1]{\left\lVert#1\right\rVert}
\providecommand{\abso}[1]{\left\lvert#1\right\rvert}

\hypersetup{
 pdfauthor={E. A. Nigsch}%,
}

\begin{document}

\author{E. A. Nigsch\footnote{Fakultät für Mathematik, Universität Wien, Oskar-Morgenstern-Platz 1, 1090 Wien, Austria; e-mail: eduard.nigsch@univie.ac.at; phone: +43 1 4277 50636 }}
\title{The functional analytic foundation\\ of Colombeau algebras}
\maketitle

\begin{abstract}
Colombeau algebras constitute a convenient framework for performing nonlinear operations like multiplication on Schwartz distributions. Many variants and modifications of these algebras exist for various applications. We present a functional analytic approach placing these algebras in a unifying hierarchy, which clarifies their structural properties as well as their relation to each other.
\end{abstract}


\section{Introduction}\label{sec_colalg}

The object of the present article is to present a functional analytic description of Colombeau algebras. We introduce a basic space (Definition \ref{basdef}) representing the fundamental idea behind Colombeau algebras, \emph{regularization of distributions}, in the most general way. This will enable us to realize these algebras as particular instances of a unifying structure.

By \emph{Colombeau algebras} one understands algebras of generalized functions introduced by J.~F.~Colombeau for the purpose of rigorously defining multiplication and other nonlinear operations on Schwartz distributions in a way suitable for applications in physics. In particular, these algebras contain the space of Schwartz distributions as a linear subspace and the algebra of smooth functions as a faithful subalgebra. This enables one to employ Colombeau algebras in many situations involving ambiguous products of distributions, resolving these ambiguities and allowing for the solution of many nonlinear problems which cannot be solved in classical distribution theory. For example, one can obtain existence and uniqueness results for large classes of nonlinear partial differential equations \cite{Biagioni, Rosinger, 0701.35042}. Furthermore, classical concepts of regularity theory and microlocal analysis can be formulated and studies as well in Colombeau algebras. We suppose the reader to have a certain familiarity with this area of research and refer to \cite{ColNew, colmult, ColElem, MOBook, GKOS} for detailed information; as for terminology and notation, we mainly follow \cite{GKOS}.

Alongside Colombeau's \emph{original} algebra $\cG^o$ the \emph{elementary} algebra $\cG^e$ and the \emph{simplified} (also called \emph{special}) algebra $\cG^s$ on open subsets of $\bR^n$ appeared \cite{ColNew,ColElem, 0601.35014}. In subsequent years many variants of these algebras emerged for various applications and in different contexts. One main line of research pursued the goal of establishing a \emph{geometric} theory of nonlinear generalized functions on manifolds, eventually aiming at nonlinear distributional geometry.

In a first step this involved the construction of the \emph{diffeomorphism invariant} local algebra $\cG^d$ \cite{found} on which the \emph{global} algebra $\hat\cG$ \cite{global} on manifolds is based, which in turn was the model for the space $\hat\cG^r_s$ of nonlinear generalized tensor fields presented in \cite{global2}.

In hindsight, the major hurdle in each step along this way was the appropriate choice of the \emph{basic space} (the space containing the representatives of generalized functions) and the \emph{testing procedure} (which determines whether a function belongs to the subsets of \emph{moderate} or \emph{negligible} functions, the quotient of which constitutes the desired algebra). This often was accomplished only after a long phase of trial and error, owing to the growing technical difficulties involved therein (cf.~the discussion in \cite[Section 2.1]{GKOS}). Consequently, these algebras have the character of formally unrelated ad-hoc extensions of the original ideas. By basing our approach on abstract \emph{smoothing operators (Section \ref{sec_smoothop})} we will place these algebras in a unifying hierarchy (Section \ref{sec_subalg}). Furthermore, requirements for the following key properties will be clarified in this setting:
\begin{enumerate}[($i$)]
  \item existence of the so-called $\sigma$-embedding of smooth functions, which allows one to split the smooth from the singular part in calculations (see Remark \ref{rem_baslo} \eqref{rem_sigma});
  \item sheaf properties (see Section \ref{sec_sheaf}); % (e.g. the algebra constructed in \cite{global2} is no sheaf)
 \item invariance of the basic spaces under diffeomorphisms and derivatives (see Theorem \ref{thm_diffs});
 \item existence of a meaningful directional derivative which is $C^\infty$-linear in the directional vector field (see Section \ref{sec_deriv}).
\end{enumerate}

Points (i)--(iii) are more of a foundational nature, while (iv) will be crucial for eventually establishing a notion of covariant derivative in a full Colombeau algebra of nonlinear generalized tensor fields. In fact, the latter is not possible in $\hat\cG^r_s$, but our results will lead the way to accomplishing this objective.

It should be mentioned that the idea of regarding generalized functions as maps from the space of smoothing operators to the space of smooth functions was also used in \cite{1204.58021,1177.46033} in order to construct global algebras of generalized functions on compact manifolds.

Finally, we remark that while we only deal with the scalar case on open subsets of $\bR^n$ here, the same constructions can be done with minor modifications also for distributions taking values in a vector space or, in the context of manifolds, with values in a vector bundle.

\section{Notation}

$\bN = \{1,2,,\dotsc\}$ is the set of natural numbers, $\bN_0 = \{ 0 \} \cup \bN$, $\bR$ is the real number field and $\bC$ the field of complex numbers. $\Omega$ will be, in general, an open subset of $\bR^n$ for some $n$. $\ev_x$ denotes the evaluation map at a point $x$ and $\id_M$ (or simply $\id$) the identity on a set $M$.

$C^\infty(\Omega)$ and $\cD(\Omega)$ are the usual spaces of smooth functions and test functions of distribution theory, $\cD'(\Omega)$ the space of distributions on $\Omega$. We write the pairing of $\cD'(\Omega)$ and $\cD(\Omega)$ as $\langle u, \varphi \rangle$ for $u \in \cD'(\Omega)$ and $\varphi \in \cD(\Omega)$. $\cE'(\Omega)$ denotes the dual of $C^\infty(\Omega)$, $\delta \in \cE'(\Omega)$ is the delta functional and $\delta_x$ its translate by $x \in \bR^n$. Note that $\basE\Omega$ will \emph{not} denote the space $C^\infty(\Omega)$ as often done in distribution theory, but the basic space introduced in Definition \ref{basdef} -- an unfortunate but historically established clash of notations.

For any functor $F$ on the category of open subsets of $\bR^n$ (or smooth manifolds) with diffeomorphisms and a given diffeomorphism $\mu$ we denote the induced action $F(\mu)$ by $\mu_*$ and its inverse $F(\mu^{-1})$ by $\mu^*$, called  pushforward and pullback along $\mu$, respectively. In general, the equivalence class in a quotient will be denoted by square brackets $[\ldots]$.

For a multiindex $\alpha \in \bN_0^n$, $\pd_x^\alpha$ denotes the partial derivative of order $\alpha$ in the $x$-variable, where we may omit the $x$ in unambiguous cases. For a vector field $X \in C^\infty(\Omega, \bR^n)$ and $f \in C^\infty(\Omega)$ (or, more generally, $f \in C^\infty(\Omega, E)$ with $E$ any locally convex space), $\Lsm_X f = \sum_i X^i \pd_i f$ denotes the directional derivative of $f$ in the direction of $X$, where the $X^i$ are the coordinates of $X$; for $\varphi \in \cD(\Omega)$, $\Lnf_X \varphi$ is the $n$-form derivative $\Lnf_X \varphi \coleq \sum_i X^i \pd_i \varphi + \Div X \cdot \varphi$, where $\Div X$ is the divergence of $X$ (this is the local expression of the Lie derivative of $n$-forms on manifolds). The inclusion $\cD(\Omega) \subseteq C^\infty(\Omega)$ will cause no confusion here (see Remark \ref{rem_derivatives}). $\overline{B_r}(x)$ denotes the closed Euclidean ball of radius $r>0$ at $x \in \bR^n$.

For two locally convex spaces $E$ and $F$, $\cL(E,F)$ denotes the space of all continuous linear mappings from $E$ to $F$. Following \cite{Schaefer} we denote by $\cL_b(E,F)$ this space endowed with the topology of bounded convergence. $E \widehat\otimes F$ denotes the completed projective tensor product of $E$ and $F$; given two linear mappings $f$ on $E$ and $g$ on $F$, $f \hat \otimes g$ denotes the extension of their tensor product to $E \widehat\otimes F$. We denote by $\cs(E)$ the set of continuous seminorms on $E$.

Concerning calculus on infinite dimensional vector spaces, we employ the setting of convenient calculus of \cite{KM}. We recall its basic definitions: a mapping $f \colon E \to F$ between two locally convex space if called smooth if it maps smooth curves into $E$ to smooth curves into $F$, i.e., if $f \circ c \in C^\infty(\bR, F)$ for all $c \in C^\infty(\bR, E)$. Of particular importance is the exponential law, which states that $f \colon E_1 \times E_2 \to F$ is smooth if the canonically associated mapping $E_1 \to C^\infty(E_2, F)$ exists and is smooth. There exists a linear smooth differentiation operator $\ud \colon C^\infty(E, F) \to C^\infty(E, \cL_b(E,F))$ satisfying the chain rule. A curve is smooth into a projective limit if and only if all its components are smooth. Finally, $C^\infty(E,F)$ denotes the space of smooth functions from $E$ to $F$, which coincides with the usual notion for finite-dimensional spaces.

\section{Smoothing kernels}\label{sec_smoothop}

At the core of our approach lies the principle that Colombeau algebras always are based on representing distributions by smooth functions. In the most general way this is accomplished by the use of \emph{smoothing operators}
$\Phi \in \cL(\cD'(\Omega), C^\infty(\Omega))$.
By the Schwartz kernel theorem these correspond exactly to \emph{smoothing kernels}\footnote{Note that the `smoothing kernels' of \cite[Definition 3.3]{global} are (smoothly parametrized) nets of smoothing kernels in our terminology which satisfy some additional properties.}
$\vec \varphi \in C^\infty(\Omega, \cD(\Omega))$
and this correspondence is a topological isomorphism \cite[Theorem 44.1 and Proposition 50.4]{Treves}
\begin{equation}\label{SOSKiso}
\cL_b(\cD'(\Omega), C^\infty(\Omega)) \cong C^\infty(\Omega, \cD(\Omega))
\end{equation}
where the latter space carries its natural topology of uniform convergence on compact sets in all derivatives. Explicitly, the correspondence between a smoothing operator $\Phi$ and a smoothing kernel $\vec\varphi$ is realized in one direction by $\Phi(u) \coleq u \circ \vec\varphi$ for $u \in \cD'(\Omega)$ and in the other direction by $\vec \varphi(x) \coleq \Phi^{t}(\delta_x)$ for $x \in \Omega$, where $\Phi^{t}$ is the transpose of $\Phi$ \cite[Theor\'eme 3]{FDVV}. Following \cite{FDVV} we write $\langle u, \vec\varphi \rangle$ for the mapping $x \mapsto \langle u, \vec\varphi(x) \rangle$ in $C^\infty(\Omega)$.

The importance of these objects lies in the fact that moderateness and negligibility of generalized functions are in practice determined by evaluating them on particular nets or sequences of smoothing kernels, called \emph{test objects}. In the simplified algebra this is merely one fixed test object \cite[Equation (1.8)]{GKOS} and this evaluation is already incorporated in the embedding, while full Colombeau algebras employ a graded set of test objects for the test. For example, in the elementary algebra $\cG^e$ on $\bR^n$ these are scaled and translated test functions having integral one and a number of vanishing moments, in the diffeomorphism invariant algebras $\cG^d$ and $\hat\cG$ these are nets of smoothing kernels with appropriate asymptotic properties (\cite[Definition 3.3]{global}, \cite{gdnew}).

\section{The general basic space}

The \emph{basic space} has been realized in many forms, for example $C^\infty(\cD(\Omega))$, $C^\infty(\Omega)^{(0,1]}$ and $C^\infty(\cD(\Omega), C^\infty(\Omega))$ (cf.~\cite{colmult,ColElem,found,global,Jelinek,colmani}), not mentioning equivalent ones obtained by an application of the exponential law \cite[Theorem 3.12]{KM} or transformations between equivalent formalisms (cf.~\cite[Section 5]{found}).
\begin{remark}\label{keinanull}
Note that the original definitions of these basic spaces originally employed the space $\cA_0(\bR^n) \coleq \{ \varphi \in \cD(\bR^n): \int \varphi=1 \}$ instead of $\cD(\Omega)$. We remove this artificial restriction and use $\cD(\Omega)$ instead because this gives more natural definitions, allows for more flexibility in testing and still gives the same algebras as long as the same tests are used.
\end{remark}

In order to find an encompassing notion for these basic spaces we point out that ultimately, as far as embedded distributions and their products are concerned, in the algebras mentioned above the tests for moderateness and negligibility \emph{always} reduce to estimating $C^\infty(\Omega)$-seminorms of expressions of the form $\langle u, \vec\varphi \rangle \in C^\infty(\Omega)$ (or products thereof) where $u$ is a distribution and $\vec\varphi$ a smoothing kernel. More specifically, one considers the asymptotic behaviour of $\langle u, \vec\varphi \rangle$ when $\vec\varphi$ converges to $\vec\delta \in C^\infty(\Omega, \cE'(\Omega))$ defined by $\vec\delta(x) = \delta_x$ in a certain way, as will be made more precise in Section \ref{sec_testing}. From this point of view it is completely natural to consider distributions as \emph{mappings on smoothing kernels} via the embedding
\begin{equation}\label{lembed}
\begin{split}
\cD'(\Omega) &\to \cL_b ( \SK\Omega, C^\infty(\Omega))\\
u &\mapsto [ \vec\varphi \mapsto \langle u, \vec\varphi \rangle ].
\end{split}
\end{equation}
It is easy to see that this is a topological embedding which is natural in the sense of category theory (it commutes with diffeomorphisms).

In order to introduce a product we furthermore imitate the original idea of Colombeau and extend the range space of this embedding by replacing linear dependence on $\vec\varphi$ by smooth dependence, which leads us to the following definition.
\begin{definition}\label{basdef}
We define the space $\SK \Omega$ of \emph{smoothing kernels on $\Omega$} and the \emph{basic space} $\cE(\Omega)$ by
\begin{align*}
 \SK \Omega &\coleq C^\infty(\Omega, \cD(\Omega)), \\
 \basE\Omega &\coleq C^\infty(\SK\Omega, C^\infty(\Omega)).
\end{align*}
\end{definition}
$\basE\Omega$ contains $\cD'(\Omega)$ and $C^\infty(\Omega)$ via the linear embeddings $\iota$ and $\sigma$ defined by
\begin{align}
 &\iota\colon \cD'(\Omega) \to \basE\Omega,
(\iota u)(\vec\varphi) \coleq \langle u, \vec\varphi \rangle,
\label{iotadef}\\
&\sigma\colon C^\infty(\Omega) \to \basE\Omega,
(\sigma f)(\vec \varphi) \coleq f.
\label{sigmadef}
\end{align}
It inherits the algebra structure from $C^\infty(\Omega)$; in particular, its product is given by
\begin{equation}\label{eproduct}
(R_1 \cdot R_2)(\vec \varphi) = R_1(\vec\varphi) \cdot R_2(\vec \varphi)\qquad (R_1,R_2 \in \basE\Omega,\ \vec\varphi \in \SK\Omega).
\end{equation}
This multiplication defines a $C^\infty(\Omega)$-algebra structure on $\basE\Omega$ for which $\sigma$ becomes an algebra homomorphism. However, $\iota|_{C^\infty(\Omega)} \ne \sigma$, a defect which will be repaired by the usual quotient construction later on.

We point out the following intuitive interpretation: in a very general sense, the process of smoothing a distribution is represented by the following bilinear hypocontinuous mapping:
\begin{align*}
 \cD'(\Omega) \times \cL_b(\cD'(\Omega), C^\infty(\Omega)) & \to C^\infty(\Omega) \\
 (u , \Phi) &\to \Phi(u).
\end{align*}
Applying the exponential law \cite[\S 40 1.(3)]{Koethe2} gives the embedding \eqref{lembed} of $\cD'(\Omega)$ into a space of mappings having smoothing operators (or, equivalently, smoothing kernels) as a variable. Elements of that space can be multiplied by forgetting their linearity and thus viewing them as elements of $\basE\Omega$. This is technically simple and lends itself easily to generalizations to the tensorial case on a manifold $M$, which will need to be constructed on the space
\[ C^\infty( \cL_b( \cD'^r_s(M), \cT^r_s(M)), \cT^r_s(M)). \]
Here, $\cD'^r_s(M)$ is the space of $(r,s)$-tensor distributions and $\cT^r_s(M)$ the space of smooth tensor fields on $M$.

Finally, we mention that our method could be applied to other distribution spaces as well: instead of $\cD'(\Omega)$ and its regular counterpart $C^\infty(\Omega)$ on could as well regard $(C^\infty(\Omega))'$ and $\cD(\Omega)$, giving rise to the basic space $C^\infty(\cD(\Omega, C^\infty(\Omega)),\cD(\Omega))$ or, in the case of tempered distributions, $\cS'(\Omega)$ and $\cO_M(\Omega)$, which give rise to $C^\infty( \cO_M (\Omega, \cS(\Omega)), \cO_M(\Omega))$, etc.

\section{\texorpdfstring{Subalgebras of $\basE\Omega$}{Subalgebras of $E$}}\label{sec_subalg}

Informally, one can say that Colombeau algebras are constructed by forgetting certain properties of distributions (like linearity) and thus view them as elements of a larger space (of smooth functions) on which more operations are defined. In this process also the sheaf property gets lost. However, another property of distributions which is invariant under multiplication can be saved: for $u,v \in \cD'(\Omega)$, $\vec\varphi \in \SK \Omega$ and $x \in \Omega$, $(\iota(u) \cdot \iota(v))(\vec\varphi)(x) = \langle u, \vec\varphi(x) \rangle \cdot \langle v, \vec\varphi(x) \rangle$ depends only on $\vec\varphi$ at the point $x$. By retaining \emph{locality} properties of distributions to a certain degree one can not only recover the basic spaces of the algebras mentioned in Section \ref{sec_colalg} but also retain this sheaf property. 
We introduce the following notions.
\begin{definition}\label{locdef}A function $R \in \basE\Omega = C^\infty(\SK\Omega, C^\infty(\Omega))$ is called
\begin{asparaenum}[($i$)]
  \item \emph{local} if for all open subsets $U \subseteq \Omega$ and  smoothing kernels $\vec\varphi, \vec\psi \in \SK\Omega$ the equality $\vec\varphi|_U=\vec\psi|_U$ implies $R(\vec\varphi)|_U=R(\vec\psi)|_U$;
  \item \emph{point-local} if for all $x \in \Omega$ and smoothing kernels $\vec\varphi, \vec \psi \in \SK\Omega$ the equality $\vec\varphi(x)=\vec\psi(x)$ implies $R(\vec \varphi)(x)=R(\vec\psi)(x)$;
  \item \emph{point-independent} if for all $x,y \in \Omega$ and smoothing kernels $\vec\varphi, \vec\psi \in \SK\Omega$ the equality $\vec\varphi(x)=\vec\psi(y)$ implies $R(\vec \varphi)(x)=R(\vec \psi)(y)$.
\end{asparaenum}
We denote by $\basEloc\Omega$, $\basEploc\Omega$ and $\basEpi\Omega$ the subsets of $\basE\Omega$ consisting of local, point-local and point-independent elements, respectively, and by $\basLloc\Omega$, $\basLploc\Omega$ and $\basLpi\Omega$ the corresponding subsets of $\basL\Omega \coleq \cL(\SK\Omega, C^\infty(\Omega))$. Finally, $\basLcinf\Omega$ denotes the subspace of $\basL\Omega$ consisting of $C^\infty(\Omega)$-linear mappings, where the $C^\infty(\Omega)$-module structure on $\SK\Omega$ is given by $(f \cdot \vec\varphi)(x) \coleq f(x) \cdot \vec\varphi(x)$ for $f \in C^\infty(\Omega)$ and $\vec \varphi \in \SK\Omega$.
\end{definition}
\begin{remark}\label{rem_cats}
\begin{enumerate}[($i$)]
\item\label{rem_cats1} Relative to $\basE\Omega$, one sees that $\basEloc\Omega$ and $\basEploc\Omega$ are sub-$C^\infty(\Omega)$-algebras, $\basL\Omega$, $\basLloc\Omega$, $\basLploc\Omega$ and $\basLcinf\Omega$ are sub-$C^\infty(\Omega)$-modules, $\basEpi\Omega$ is a sub-$\bC$-algebra and $\basLpi\Omega$ is a linear subspace.
 \item $R \in \basL\Omega$ obviously is point-local if and only if for all $\vec\varphi \in \SK\Omega$ and $x \in \Omega$, $\vec\varphi(x)=0$ implies $R(\vec\varphi)(x)=0$, and local if and only if for all $\vec\varphi \in \SK\Omega$ and $U \subseteq \Omega$ open, $\vec\varphi|_U=0$ implies $R(\vec\varphi)|_U=0$, which is equivalent to $\supp R(\vec\varphi) \subseteq \supp \vec\varphi$.
 \item Many more notions giving rise to more refined subalgebras of $\basE\Omega$ can be thought of. For example, one could as well incorporate derivatives in the definition of point-locality, giving rise to a notion of \emph{k-jet locality} where $R(\vec\varphi)(x)$ depends only on the derivatives $\pd_x^\alpha \vec\varphi(x)$ of order $\abso{\alpha} \le k$. Potentially, a result similar to the Peetre theorem can be obtained here, resulting in locally finite order of elements of $\basEloc\Omega$. We will not examine this notion further here; however, for applications one should keep in mind the general principle that the basic space can be adapted to specific problems by restricting it to a subalgebra which is invariant under the operations one wishes to perform.
\end{enumerate}
\end{remark}

The assignment $\Omega \mapsto \basE\Omega$ defines a functor from the category of open subsets of some $\bR^n$ (or more generally, smooth manifolds) with diffeomorphisms into the category of locally convex spaces with bounded (equivalently smooth) linear mappings. This means that for any diffeomorphism $\mu\colon\Omega \to \Omega'$ we have an induced action $\mu_*\colon \basE\Omega \to \basE{\Omega'}$ given by $(\mu_* R)(\vec \varphi) \coleq \mu_*(R(\mu^*\vec\varphi)) = \mu_*(R(\mu^* \circ \vec\varphi \circ \mu))$. Because the properties of Definition \ref{locdef} are invariant under this action the definition of the respective subspaces also is functorial and the following inclusions, which are easily seen to be proper, are natural:
\begin{equation}\label{inclusions}
\begin{gathered}
\xymatrix@R=1em@C=1em{
\basE\Omega \ar@{}[r]|-*[@]\txt{$\supseteq$} & \basEloc\Omega \ar@{}[r]|-*[@]\txt{$\supseteq$} & \basEploc\Omega \ar@{}[r]|-*[@]\txt{$\supseteq$} &  \basEpi\Omega \\
\ar@{}[u]|-*[@]\txt{$\subseteq$} \ar@{}[r]|-*[@]\txt{$\supseteq$} \basL\Omega &\ar@{}[u]|-*[@]\txt{$\subseteq$} \ar@{}[r]|-*[@]\txt{$\supseteq$}   \basLloc\Omega &\ar@{}[u]|-*[@]\txt{$\subseteq$} \ar@{}[r]|-*[@]\txt{$\supseteq$}   \basLploc\Omega & \ar@{}[u]|-*[@]\txt{$\subseteq$} \basLpi\Omega
}
\end{gathered}
\end{equation}
Note that $\iota$ actually maps into $\basLpi\Omega$ and hence can be seen as a map into each of the above spaces, depending on the context. 

In the following proof and also later on, we use the following notation: given $\varphi \in \cD(\Omega)$, we denote by $\tilde\varphi$ the constant function $\tilde \varphi \coleq [x \mapsto \varphi] \in \SK \Omega$. We will also write $\varphi^\sim$ instead of $\tilde\varphi$.
\begin{proposition}\label{isos}We have the following natural isomorphisms:
\begin{enumerate}[($i$)]
\item\label{prop4.1} $\basEploc\Omega \cong C^\infty(\cD(\Omega), C^\infty(\Omega))$;
\item\label{prop4.2} $\basEpi\Omega \cong C^\infty(\cD(\Omega))$;
\item\label{prop4.3} $\basLploc\Omega\cong \basLcinf\Omega \cong C^\infty(\Omega, \cD'(\Omega))$;
\item\label{prop4.4} $\basLpi\Omega \cong \cD'(\Omega)$.
\end{enumerate}
\end{proposition}
\begin{proof}
For \eqref{prop4.1} we have the correspondence 
\begin{equation}\label{alphaa}
\begin{split}
R \in \basEploc\Omega &\cong C^\infty(\cD(\Omega), C^\infty(\Omega)) \ni S\\
R(\vec\varphi)(x) & \coleq S(\vec\varphi(x))(x) \\
S(\varphi)(x) & \coleq R(\tilde \varphi)(x)
\end{split}
\end{equation}
for $\vec\varphi \in \SK\Omega$, $x \in \Omega$ and $\varphi \in \cD(\Omega)$.
For \eqref{prop4.2} we have the correspondence
\begin{align*}
 R \in \basEpi\Omega &\cong C^\infty(\cD(\Omega)) \ni S \\
R(\vec\varphi)(x) & \coleq S(\vec\varphi(x)) \\
S(\varphi) &\coleq R(\tilde \varphi) \in \bC \subseteq C^\infty(\Omega)
\end{align*}
for $\vec\varphi \in \SK\Omega$, $x \in \Omega$ and $\varphi \in \cD(\Omega)$, noting that $R(\tilde \varphi)$ is constant.

For \eqref{prop4.3}, given $R \in \basLploc\Omega$, we have $R(f \cdot \vec\varphi)(x) = R(f(x)\cdot \vec\varphi)(x) = f(x) \cdot R(\vec\varphi)(x) = (f \cdot R(\vec \varphi))(x)$ for $f \in C^\infty(\Omega)$, $\vec\varphi \in \SK\Omega$ and $x \in \Omega$, hence $R \in \basLcinf\Omega$. For the converse we identify $\SK\Omega$ with $C^\infty(\Omega) \widehat\otimes \cD(\Omega)$ (see \cite{Treves,FDVV} for details) and claim that the following diagram commutes for all $R \in \basLcinf\Omega$ and $x \in \Omega$:
\[
 \xymatrix@C=4em{
\ar[d]_{R} C^\infty(\Omega) \widehat\otimes \cD(\Omega) \ar[r]^-{\ev_x \widehat\otimes \id} & \bC \widehat\otimes \cD(\Omega) \cong \cD(\Omega) \ar[d]^{\varphi \mapsto 1 \otimes \varphi} \\
C^\infty(\Omega) \ar[d]_{\ev_x} & C^\infty(\Omega)\widehat\otimes \cD(\Omega)\ar[d]^{R} \\
\bC & C^\infty(\Omega) \ar[l]^{\ev_x}
}
\]
In fact, because all maps are continuous it suffices to check commutativity on elements of $C^\infty(\Omega)\otimes\cD(\Omega)$:
\begin{multline*}
 R(1 \otimes ( \ev_x \otimes \id)(f \otimes \varphi))(x)  = R(1 \otimes(f(x) \otimes \varphi))(x) = R(1 \otimes (f(x)\cdot \varphi))(x) \\
= f(x) \cdot R(1 \otimes \varphi)(x) = (f \cdot R(1 \otimes \varphi))(x) = R(f \otimes \varphi)(x).
\end{multline*}
Consequently, $\vec\varphi(x)=0$ implies $(\ev_x \widehat\otimes \id) \vec\varphi = 0$ and hence $(\ev_x \circ R)(\vec\varphi) = 0$, which means that $R \in \basLploc\Omega$.

$\basLploc\Omega \cong C^\infty(\Omega, \cD'(\Omega))$ follows as in \eqref{prop4.1} and \eqref{prop4.4} as in \eqref{prop4.2}. That the given isomorphisms result in smooth mappings follows from the exponential law \cite[3.12]{KM}.
\end{proof}

From the above proof we note that
\begin{align}
R \in \basEploc\Omega &\Rightarrow R(\vec\varphi)(x) = R(\vec\varphi(x)^\sim)(x) \label{plocform} \\
R \in \basEpi\Omega &\Rightarrow R(\vec\varphi)(x) = R(\vec\varphi(x)^\sim) \label{piform}
\end{align}
for all $\vec\varphi \in \SK\Omega$ and $x \in \Omega$.

\begin{remark}\label{rem_baslo}\begin{enumerate}[($i$)]
\item Using the isomorphisms of Proposition \ref{isos} one sees, taking  Remark \ref{keinanull} into account, that $\basEpi\Omega$ is the basic space of $\cG^o$ originally used by Colombeau \cite{ColNew} while $\basEploc\Omega$ was used (up to an application of the exponential law) in the construction of $\hat\cG$ \cite{found, global}. Other basic spaces can be obtained by considering the pullback of $\basE\Omega$ along a map into $\SK\Omega$; for example, the basic space of $\cG^s$ is obtained as the pullback of $\basE\Omega$ along $k \mapsto \vec\psi_k$ for a specific sequence of smoothing kernels $(\vec\psi_k)_k$ (see Section \ref{sec_realize}).
This underlines the significance of $\basE\Omega$ at the base of a unified understanding of Colombeau algebras.
\item \label{rem_sigma}The mapping $\sigma$ defined by \eqref{sigmadef} actually sends $C^\infty(\Omega)$ into $\basEploc\Omega$ but not into $\basEpi\Omega$ nor in any of the spaces of linear maps because $\sigma f$ is not linear in $\vec\varphi$. Hence, the $\sigma$-embedding only exists for the spaces $\basE\Omega$, $\basEloc\Omega$ and $\basEploc\Omega$. The use of this embedding lies in the fact that it embeds smooth functions as they are, without any regularization involved.
\item Proposition \ref{isos} clarifies the role, in obtaining sheaf properties, of coupling all occurrences (in various slots) of the space variable when testing. For example, in $\cG^d$ an element $R$ of the basic space $C^\infty(\cD(\Omega), C^\infty(\Omega))$ is tested by an expression of the form
\[ R(\vec \varphi_\e(x))(x) \]
with $\vec\varphi_\e \in \SK\Omega$ and $x \in \Omega$ (cf.~\cite{gdnew}). That $x$ appears both in the codomain slot of $R$ and the first subslot of its domain is essential for obtaining localizability, as it forces the support of $\vec\varphi_\e(x)$ to be near $x$ if $(\vec\varphi_\e)_\e$ is localizing as defined below (Definition \ref{def_localnet}). This coupling now can be formally understood to come about by the isomorphism \eqref{alphaa}, which states exactly that evaluating $R$ on a smoothing kernel $\vec\varphi_\e$ naturally gives the expression $R(\vec\varphi_\e(x))( x)$.
\end{enumerate}
\end{remark}

\section{Testing}\label{sec_testing}

The second main constituent of Colombeau algebras besides the basic space is the testing procedure, whose underlying principles we will describe now in functional analytic terms. We recall that, from Colombeau's original algebra $\cG^o$ up to and including the latest tensorial algebra $\hat\cG^r_s$, the proper character and intuitive understanding of the test objects was a source of many difficulties. Although a comprehensive study of these test objects is beyond the scope of this article (see \cite{found} for one possible direction to follow), we will list the essential characteristics they need to have. This does not characterize their variety at large, but at least we obtain a conceptually simple framework for adapting the construction to various settings.

The aim of testing is to \emph{identify} $\iota f$ and $\sigma f$ for $f \in C^\infty(\Omega)$ by a quotient construction and hence turn the restriction of $\iota$ to $C^\infty(\Omega)$ into an algebra homomorphism. In other words, the product \eqref{eproduct} will then preserve the product of smooth functions in the quotient. The key mechanism behind this is based on the following observation, for which we use that, similar to the case of smoothing kernels as in \eqref{SOSKiso}, there is an isomorphism
\[ \cL_b( C^\infty(\Omega), C^\infty(\Omega)) \cong C^\infty(\Omega, \cE'(\Omega)) \]
under which the identity on $C^\infty(\Omega)$ corresponds to $\vec\delta \coleq [ x\mapsto \delta_x ]$.

\begin{proposition}\label{prop_testing}
 Let $u \in \cD'(\Omega)$ and $f \in C^\infty(\Omega)$. Choose a sequence $(\vec\varphi_k)_k$ in $\SK\Omega$ such that for the corresponding sequence $(\Phi_k)_k$ in $\cL(\cD'(\Omega), C^\infty(\Omega))$ we have
 \begin{equation}\label{smoothapprox}
  \Phi_k|_{C^\infty(\Omega)} \to \id_{C^\infty(\Omega)}\textrm{ in }\cL_b(C^\infty(\Omega), C^\infty(\Omega))
 \end{equation}
  as well as
  \begin{equation}\label{distrapprox}
\Phi_k \to \id_{\cD'(\Omega)}\textrm{ in }\cL_b(\cD'(\Omega), \cD'(\Omega)).
\end{equation}
Then $u = f$ if and only if $(\iota u - \sigma f )(\vec\varphi_k) \to 0 \textrm{ in }C^\infty(\Omega)$.
\end{proposition}
\begin{proof}
 Assuming $u=f$, $\iota f$ has a continuous extension to $C^\infty(\Omega, \cE'(\Omega))$, which can be seen from
 \begin{align*}
 C^\infty(\Omega) & \cong \cL_b(\cE'(\Omega)) \subseteq \cL_b(\cE'(\Omega) \widehat\otimes C^\infty(\Omega), C^\infty(\Omega))\\
&\cong \cL_b(C^\infty(\Omega, \cE'(\Omega)), C^\infty(\Omega)),
\end{align*}%
where the inclusion is given by $u \mapsto u \hat\otimes \id_{C^\infty(\Omega)}$. Hence, $(\iota - \sigma)(f)(\vec \varphi_k) \to (\iota f)(\vec\delta) - f = 0$. Conversely, $(\iota u)(\vec\varphi_k) = \Phi_k(u) \to u$ in $\cD'(\Omega)$ and $(\iota u)(\vec \varphi_k) \to f$ in $C^\infty(\Omega)$ and thus in $\cD'(\Omega)$ implies $u=f$. 
\end{proof}

Hence, in the sum of the subspaces $\iota(\cD'(\Omega))$ and $\sigma(C^\infty(\Omega))$ in $\basE\Omega$, $(\iota-\sigma)(C^\infty(\Omega))$ can be \emph{characterized} as the set of those $R$ satisfying $R(\vec\varphi_k) \to 0$. This characterization does not hold in the whole space $\basE\Omega$, however, and in order to obtain an ideal containing $(\iota - \sigma)(C^\infty(\Omega))$ one needs to employ some kind of asymptotic scale (cf.~\cite{0919.46027,1052.46002,1126.46030}) to ensure that for $R(\vec\varphi_k) \to 0$ and given $u \in \cD'(\Omega)$, $R(\vec \varphi_k)\cdot (\iota u)(\vec\varphi_k)$ still converges to zero. 
There are many possible choices of sequences $(\vec\varphi_k)_k$ for this purpose;  we will give an example resembling the construction of the classical full Colombeau algebras.

For this we introduce a grading on the test objects; a sequence $(\vec\varphi_k)_k$ shall be said to be of oder $q \in \bN_0$ if \eqref{smoothapprox} converges like $O(k^{-(q+1)})$ measured with respect to continuous seminorms in that space; Additionally, one also requires suitable estimates on the growth of $\Phi_k(u)$ for $u \in \cD'(U)$ in order for embedded distributions to be moderate. Formulating this in terms of continuous seminorms of the respective spaces we obtain the following definition of test objects of order $q$: they are sequences $(\vec\varphi_k)_k \in \SK\Omega^\bN$ satisfying
\begin{enumerate}[($i$)]
 \item \label{prop1} $\forall p \in \cs ( \cL_b(C^\infty(\Omega), C^\infty(\Omega ))$: $p ( \Phi_k|_{C^\infty(\Omega)} - \id_{C^\infty(\Omega)} ) = O(k^{-(q+1)})$.
 \item \label{prop3} $\forall p \in \cs ( \cL_b(\cD'(\Omega), C^\infty(\Omega))$ $\exists N \in \bN$: $p ( \Phi_k ) = O(k^N)$
 \item \label{prop2} $\forall p \in \cs ( \cL_b(\cD'(\Omega), \cD'(\Omega))$: $p ( \Phi_k - \id_{\cD'(\Omega)} ) \to 0$.
\end{enumerate}
All Colombeau algebras mentioned in Section \ref{sec_colalg} employ tests using such sequences (or more generally, nets indexed by $(0,1]$) of smoothing kernels or variants thereof.

It is instructive to translate these conditions into analytic terms: \eqref{prop1} means that for any bounded subset $B \subseteq C^\infty(\Omega)$, any compact set $K \subseteq \Omega$ and any multiindex $\alpha \in \bN_0^n$,
\begin{equation}\label{sim0}
\sup_{x \in K, f \in B} \abso{ \int f(y) \pd_x^\alpha \vec\varphi_k(x)(y)\,\ud y - \pd^\alpha f(x) } = O(k^{-(q+1)})\qquad (k \to \infty).
\end{equation}
For \eqref{prop3}, a continuous seminorm can be assumed to be of the following form with $B \subseteq \cD'(\Omega)$ bounded, $K \subseteq \Omega$ compact and $k \in \bN_0$:
\begin{align*}
 p(\Phi) &= \sup_{x \in K, u \in B,\abso{\alpha} \le k} \abso{ \pd_x^\alpha \Phi(u)(x) } = \sup_{x \in K, u \in B, \abso{\alpha}\le k} \abso{\langle u, \pd_x^\alpha \vec\varphi(x)\rangle }
\end{align*}
Because bounded subsets of $\cD'(\Omega)$ are equicontinuous and $\{ \pd_x^\alpha \vec\varphi(x)\ |\ x \in K, \abso{\alpha} \le k \}$ is bounded in $\cD'(\Omega)$, this can be estimated by
\begin{equation}\label{sim1}
p(\Phi) \le \sup \{ \abso{\pd_x^\alpha \pd_y^{\beta} \vec\varphi(x)(y)}\ |\ x \in K, y \in \Omega, \abso{\alpha} \le k, \abso{\beta} \le l \}
\end{equation}
for some $l \in \bN$. It is remarkable to see that properties \eqref{sim0} and \eqref{sim1} already come very close to those of the test objects used for $\cG^d$ and $\hat\cG$ (cf.~\cite{gdnew}), which were found on a completely different route (see \cite[Chapter 2]{GKOS}). Condition \eqref{prop2} implies that the embedded image of a distribution converges weakly to that  distribution in the sense of \emph{association}, which is a central property to have for compatibility with classical distribution theory and in fact a cornerstone of Colombeau theory. Finally, \eqref{prop1} together with \eqref{prop3} imply that the product of a smooth function and a distribution in the Colombeau algebra is associated to their classical product.

Note that in $\cG^d$ a stronger version of \eqref{prop3} was used, which together with the \emph{localizing} property of smoothing kernels (see below) implies (iii). We have turned the tables here and required \eqref{prop2} from the beginning, which enables one to have the natural formulation of \eqref{prop3} given here. This simplifies many things. As a matter of fact, these three conditions give a very natural and easy way to the construction of a diffeomorphism invariant Colombeau algebra, which is a sheaf if we also require the test objects to be localizing as defined below. We will sketch its construction now.

First, denote the space of test objects of order $q$ as defined above by $\widetilde\cA_q(U)$.
Then, call $R \in \basEloc U$ (one could also use $\SbasEploc$ or $\SbasEpi$ instead) \emph{moderate} if $\forall p \in \cs ( C^\infty(U) )$ $\exists N,q \in \bN_0$ $\forall (\vec\varphi_k)_k \in \widetilde\cA_q(U)$: $p ( R ( \vec\varphi_k ) ) = O(k^{N})$, and \emph{negligible} if $\forall p \in \cs ( C^\infty(U))$ $\forall m \in \bN$ $\exists q \in \bN_0$ $\forall (\vec\varphi_k)_k \in \widetilde\cA_q(U)$: $p(R(\vec\varphi_k)) = O(k^{-m})$. The quotient of moderate by negligible elements then satisfies all the requiremens for a Colombeau algebra, as is easily verified. Furthermore and most importantly, the resulting algebra is automatically diffeomorphism invariant: because the properties of the test objects are formulated in terms of continuous seminorms and because diffeomorphisms acting on smooth functions and distributions are linear continuous maps, the spaces of test objects as well as moderateness and negligiblity are immediately seen to be invariant under the action of diffeomorphisms. The sheaf property is obtained if we require the test objects to be localizing in addition, as is seen in Section \ref{sec_sheaf}

This formulation of the basic spaces and the testing procedure hence pays off quickly and saves one from a big part of the technicalities hitherto involved in constructing a diffeomorphism invariant full algebra. We will apply these ideas in an upcoming paper in order to construct an algebra of generalized sections of vector bundles on manifolds.

\section{Sheaf properties} \label{sec_sheaf}

$\cD'$ is well known to be a sheaf. In this section we will study in which sense the spaces of diagram \eqref{inclusions} can be turned into sheaves. We first introduce a restriction mapping on $\SbasEloc$.

\begin{theorem}\label{eloc_restriction}Let $U,V \subseteq \Omega$ be open, $V \subseteq U$ and $R \in \basEloc{U}$. Then there is a \emph{unique} element $R|_V \in \basEloc{V}$ such that
\begin{itemize}
 \item[(*)] for any open set $W$ with $\overline{W} \subseteq V$, $\vec\varphi \in \SK V$ and $\vec\psi \in \SK U$ such that $\vec\varphi|_W = \vec\psi|_W$ we have $R|_V (\vec\varphi)|_W = R(\vec\psi)|_W$.
\end{itemize}
The mapping $R \mapsto R|_V$ preserves linearity and the locality properties of Definition \ref{locdef}. Furthermore, for $U,V,W \subseteq \Omega$ open with $W \subseteq V \subseteq U$ and $R \in \basEloc U$,  $(R|_V)|_W = R|_W$.
\end{theorem}

\begin{proof}
Given $\vec\varphi \in \SK V$ we define, for any open set $W$ with $\overline{W} \subseteq V$, $f_W \coleq R(\vec\psi)|_W$ where $\vec\psi \in \SK U$ is chosen such that $\vec\psi|_W = \vec\varphi|_W$. For this, take a smooth function $\rho_W \in C^\infty(V)$ with support in a closed subset of $V$ and $\rho_W \equiv 1$ on $W$, and let $\vec\psi$ be the trivial extension by zero of $\rho_W \cdot \vec\varphi$ to $U$. Because of locality of $R$, $f_W$ does not depend on the choice of $\vec\psi$.

Covering $V$ by such sets $W$ we obtain a family $(f_W)_W$ such that for any two $W_1,W_2$ with corresponding $\vec\psi_1$, $\vec\psi_2 \in \SK U$, $f_{W_1}|_{W_1 \cap W_2} = R(\vec\psi_1)|_{W_1 \cap W_2} = R(\vec\psi_2)|_{W_1 \cap W_2} = f_{W_2}|_{W_1 \cap W_2}$ because $\vec\psi_1|_{W_1 \cap W_2} = \vec\varphi|_{W_1 \cap W_2} = \vec\psi_2|_{W_1 \cap W_2}$. Hence, there exists a unique element $f \in C^\infty(U)$ such that $f|_W = f_W$ for all $W$. We set $R|_V(\vec\varphi) \coleq f$.

By \cite[3.8]{KM}, $R \in C^\infty(\SK V, C^\infty(V))$ if $|_W \circ R \in C^\infty(\SK V, C^\infty(W))$ for all $W$ as above, which is the case if for $c \in C^\infty(\bR, \SK V)$, $|_W \circ R \circ c \in C^\infty(\bR, C^\infty(W))$. The latter map is given by
\begin{equation}\label{heiis}
 R(c(t))|_W = R ( \rho_W \cdot c(t))|_W
\end{equation}
where $\rho_W \cdot c(t)$ is trivially extended to a map in $C^\infty(U, \cD(U))$. The result depends smoothly on $t$, so \eqref{heiis} is smooth.

$R|_V$ is local: suppose $\vec\varphi|_W = \vec\psi|_W$ for $\vec\varphi, \vec\psi \in \SK V$ and $W \subseteq V$ open. Then for any open set $X$ with $\overline{X} \subseteq W$, $R|_V(\vec\varphi)|_X = R(\rho_X \cdot \vec\varphi)|_X = R(\rho_X \cdot \vec\psi)|_X = R|_V ( \vec\psi)|_X$ for suitable $\rho_X$, hence $R|_V(\vec\varphi)|_W = R|_V(\vec\psi)|_W$. If $R$ is point-local or point-independent, the same obviously holds for $R|_V$.

As to uniqueness, assume any $S \in \basEloc V$ satisfies (*). Then for $\vec\varphi \in \SK V$, $W$ as above and $\vec\psi \in \SK U$ with $\vec\varphi|_W = \vec\psi|_W$, $S(\vec\varphi)|_W = R(\vec\psi)|_W = R|_V(\vec\psi)|_W$ and hence $S = R$.

For transitivity, let $X \subseteq W$ be open with $\overline{X} \subseteq W$, $\vec\varphi \in \SK W$, $\vec\psi_1 \in \SK V$ such that $\vec\psi_1|_X = \vec\varphi|_X$ and $\vec\psi_2 \in \SK U$ such that $\vec\psi_2|_X = \vec\psi_1|_X$. Then also $\vec\psi_2|_X = \vec\varphi|_X$ and thus $(R|_V)|_W(\vec\varphi)|_X = R|_V(\vec\psi_1)|_X = R(\vec\psi_2)|_X = R|_W ( \vec\varphi)|_X$.
\end{proof}

This construction fails for $\basE\Omega$ and $\basL\Omega$, however. Furthermore, elements of $\basEloc\Omega$ and its subspaces (except for $\cD'(\Omega)$) are not uniquely defined by their restrictions to the elements of an open covering of $\Omega$. This is seen by the following examples which, by the inclusion relations of \eqref{inclusions}, also settle the case of $\basEploc\Omega$ and $\basLloc\Omega$.

\begin{example}
\begin{enumerate}[($i$)]
 \item Take $\delta_a \cdot \delta_b \in \basEpi\Omega$ with $a \ne b$ and cover $\Omega$ by open sets such that none of them contains both $a$ and $b$; then the restrictions to these open sets are zero, but $\delta_a \cdot \delta_b$ is not.
\item Let $\Omega = \R$, $f \in C^\infty(\R)$ with $f(x)=0$ for $x \le 0$, $f(x)=1$ for $x \ge 1$ and define $R \in \basLploc\Omega \cong C^\infty(\Omega, \cD'(\Omega))$ by $R(x) \coleq f(x) \cdot \delta + (1 - f(x))\cdot \delta_1$. Then $R|_{(-\infty, 1)} = R|_{(0, \infty)} = 0$ but $R(x)(\varphi) = f(x)\varphi(0) + (1-f(x))\varphi(1)$ is not necessarily zero.
\end{enumerate}
\end{example}

Hence, we have shown that $\SbasEloc$, $\SbasEploc$ and $\SbasEpi$ as well as $\SbasLloc$ and $\SbasLploc$ are only presheaves. At first sight this contrasts the fact that all Colombeau algebras mentioned in Section \ref{sec_colalg} are sheaves. A closer look at the testing procedure will make clear how this comes about.

As seen in the section on testing, moderateness and negligibility of $R \in \basEloc U$ are determined by examining the asymptotic behaviour of $p(R(\vec\varphi_k))$, where $p$ is a continuous seminorm of $C^\infty(U)$ and $(\vec\varphi_k)_k \in \SK U^\bN$ a sequence of smoothing kernels (simply called \emph{test object}). Because a basis of continuous seminorms of $C^\infty(U)$ is given by the $\sup$-norms on compact sets of all derivatives, only the local asmyptotic behaviour of $R(\vec\varphi_k)$ plays a role in the test. This in turn, by locality of $R$, only depends on the local asymptotic behaviour of $(\vec\varphi_k)_k$. Hence, one is naturally led to the definition of the following quotients.

\begin{definition}
For any locally convex space $E$ we denote by $\Cinfnets{U, E}$ the quotient algebra $C^\infty(U, E)^\bN / N(U,E)$, where $N(U,E)$ is the ideal consisting of all sequences $(f_k)_k \in C^\infty(U, E)^\bN$ such that $\forall x \in U$ $\exists$ an open neighborhood $V$ of $x$ contained in $U$ $\exists k_0 \in \bN$ $\forall k \ge k_0$: $f_k|_V = 0$. We write $(f_k)_k \sim (g_k)_k$ if $(f_k)_k - (g_k)_k \in N(U,E)$ and denote by $[(f_k)_k]$ the class of $(f_k)_k$ in $\Cinfnets{U,E}$.

In particular, we set $\Cinfnets U \coleq \Cinfnets{U, \bC}$ and $\SKnets U \coleq \widetilde{C^\infty}(U, \cD(U))$.
\end{definition}
Note that $\Cinfnets{U}$ is a sheaf, as is easily seen (actually, it is the sheafification of the presheaf of sequences of smooth functions which are globally eventually equal).
Because $(\vec\varphi_k)_k \sim (\vec\psi_k)_k$ implies $(R(\vec\varphi_k))_k \sim (R(\vec\psi_k))_k$, elements of $\basEloc U$ can be regarded as maps from $\SKnets U$ to $\Cinfnets U$:
\begin{definition}For $R \in \basEloc U$ and $[(\vec\varphi_k)_k] \in \SKlocnets{U}$ we set $R ( [ ( \vec\varphi_k)_k ] ) \coleq [ ( R ( \vec\varphi_k ) )_k ] \in \widetilde{C^\infty}(U)$.
\end{definition}

It is clear now that the usual tests for moderateness and negligibility, as for example those given in Section \ref{sec_testing}, can be formulated in terms of the values of $R \in \basEloc U$ on elements of $\SKnets U$: we can write $p(R(\vec\varphi_k)) = O(k^N) \Leftrightarrow p(R([(\vec\varphi_k)_k])) = O(k^N)$ because $p$ maps $\Cinfnets \Omega$ into the space of sequences of real numbers modulo eventually equal ones. Now the sheaf property in Colombeau algebras usually comes about because in their construction, one only tests with particular sequences of smoothing kernels, namely, such that for each $x \in \Omega$ the support of $\vec\varphi_k(x)$ is eventually contained in any neighborhood of $x$ \cite[Definition 3.3.5 (i)]{GKOS}:

\begin{definition}\label{def_localnet} A sequence $(\vec\varphi_k)_k \in \SK\Omega^\bN$ is called \emph{localizing} if $\forall x \in \Omega$ $\exists$ an open neighborhood $V$ of $x$ contained in $U$ $\forall r>0$ $\exists k_0 \in \bN$ $\forall k \ge k_0$ $\forall x \in V$: $\supp \vec\varphi_k(x) \subseteq \overline{B_r}(x)$.

We denote by $\SKloc U$ the set of all localizing sequences of smoothing kernels on $U$ and by $\SKlocnets U$ the quotient $\SKloc U / \sim$ (which is given by the elements of $\SKnets U$ having localizing representatives).
\end{definition}

We will now show that $\SSKlocnets$ is a sheaf; this will play a central role in proving the sheaf property of $\SbasEloc$. The first step is the presheaf structure:

\begin{theorem}\label{sk_restriction}
For any open sets $V \subseteq U$ open there exists a linear continuous map
$\rho^{SK}_{V,U}\colon \SK U \to \SK V$
such that the following holds.
\begin{enumerate}[(i)]
\item\label{9.0} Given $(\vec\varphi_k)_k \in \SKloc U$, $(\rho^{SK}_{V,U} \vec\varphi_k)_k \in \SKloc V$.
\item\label{9.1} For $(\vec\varphi_k)_k \in \SKloc{U}$, $[(\rho^{SK}_{V,U} \vec\varphi_k)_k] = [(\vec\varphi_k|_V)_k]$ in $\widetilde{C^\infty}(V, \cD(U))$. In particular, $(\vec\varphi_k)_k \sim 0$ implies $(\rho^{SK}_{V,U} \vec\varphi_k)_k \sim 0$, hence $\rho^{SK}_{V,U}$ induces a map $|_V \colon \SKlocnets U \to \SKlocnets V$.
\item\label{9.1.1} For $\vec\varphi \in \SK U$ and $f \in C^\infty(U)$, $\rho^{SK}_{V,U} ( f \cdot \vec\varphi) = f|_V \cdot \rho^{SK}_{V,U}(\vec\varphi)$.
\item\label{9.2} For $\vec\varphi, \vec \psi \in \SK U$ and $x \in V$, $\vec\varphi(x) = \vec\psi(x)$ implies $\rho^{SK}_{V,U}(\vec\varphi)(x) = \rho^{SK}_{V,U}(\vec\psi)(x)$. % point-independence... das hat man nur für lokalisierende für große k, und zwar durch (ii).
\item\label{9.3} For open sets $U_2 \subseteq U_1 \subseteq U$ we have $|_{U_2} \circ |_{U_1} = |_{U_2}$ on $\SKlocnets U$.
\end{enumerate}
\end{theorem}
\begin{proof}
Cover $V$ by open subsets $W$ such that $\overline{W}$ is compact and contained in $V$. Choose a partition of unity $(\chi_W)_W$ on $V$ subordinate to the $W$'s and functions $\theta_W \in C^\infty(U)$ for each $W$ such that $\theta_W \equiv 1$ on an open neighborhood of $\supp \chi_W$ and $\supp \theta_W \subseteq V$.

Define $\rho^{SK}_{V,U} \colon \SK U \to \SK V$ by
\[ (\rho^{SK}_{V,U}\vec\varphi)(x) \coleq \sum_W \chi_W(x) \cdot (\vec\varphi (x) \cdot \theta_W)|_V \qquad (x \in V). \]
Note that $(\vec\varphi(x) \cdot \theta_W)|_V \in \cD(V)$, hence $\rho^{SK}_{V,U}\vec\varphi \in \SK V$. \eqref{9.1.1} and \eqref{9.2} are clear from this definition.

Any $x \in V$ has an open neighborhood $X$ contained in $V$ and intersecting only finitely many $\supp \chi_W$, say those for $W=W_1, \dotsc, W_n$. Then we have
\begin{equation}\label{blah}
(\rho^{SK}_{V,U} \vec\varphi)(x) = \sum_{i=1}^n \chi_{W_i}(x) \cdot (\vec\varphi(x) \cdot \theta_{W_i})|_V \qquad(x \in X).
\end{equation}
Because $C^\infty(V, \cD(V))$ carries the projective topology with respect to all restrictions $|_X$ and \eqref{blah} is a sum of linear continuous maps, $\rho^{SK}_{V,U}$ is continuous.

\eqref{9.0}: Given $(\vec\varphi_k)_k \in \SKloc{U}$, from \eqref{blah} it is clear that $ \supp (\rho^{SK}_{V,U} \vec\varphi_k)(x) \subseteq \supp \vec\varphi_k(x)$ for all $x \in V$, thus $(\rho^{SK}_{V,U} \vec\varphi_k)_k$ is localizing.

Because each $\supp \chi_W$ is compact we can for each $i=1 \dotsc n$ choose $k_i \in \bN$ such that for $k \ge k_i$ and $x \in \supp \chi_{W_i}$, $\theta_{W_i} \equiv 1$ on $\supp \vec\varphi_k(x)$. Consequently, for $k$ larger than all $k_i$ we have $(\rho^{SK}_{V,U} \vec\varphi_k)(x) = \sum_{i=1}^n \chi_{W_i} \cdot \vec\varphi_k(x) = \vec\varphi_k(x)$ for $x \in X$, which is \eqref{9.1}.

For \ref{9.3} we only have to note that for all open sets $X$ which are relatively compact in $U_2$, by (i) we have $(\rho^{SK}_{U_2,U_1} \rho^{SK}_{U_1,U} \vec\varphi_k)|_X = (\rho^{SK}_{U_1,U}\vec\varphi_k)|_X = \vec\varphi_k|_X = (\rho^{SK}_{U_2,U}\vec\varphi_k)|_X$ for large $k$.
\end{proof}

\begin{proposition}\label{sk_sheaf}
 $\StSKloc$ is a sheaf of $C^\infty$-modules on $\Omega$.
\end{proposition}
\begin{proof}
Let $U \subseteq \Omega$ be open, $(U_\lambda)_\lambda$ an open cover on $U$ and $[(\vec\varphi_k)_k] \in \SKlocnets{U}$. Supposing that $[ (\vec\varphi_k)_k ]|_{U_\lambda} = 0$ for each $\lambda$, we have to show that $[(\vec\varphi_k)_k]=0$. Any $x \in U$ is contained in some $U_\lambda$ and by Theorem \ref{sk_restriction} \eqref{9.1} has an open neighborhood $W$ in $U_\lambda$ such that $\vec\varphi_k|_W = (\rho^{SK}_{U_\lambda,U} \vec\varphi_k)|_W$ for large $k$. Choosing $W$ such that $\overline{W}$ is compact and contained in $U_\lambda$, $(\rho^{SK}_{U_\lambda, U} \vec\varphi_k)|_W = 0$ for large $k$ by assumption, which gives the claim.

Now let $[(\vec\varphi_k^\lambda)_k] \in \SKlocnets{U_\lambda}$ be given for each $\lambda$, satisfying $[ (\vec\varphi_k^\lambda)_k]|_{U_\lambda \cap U_\mu} = [ (\vec\varphi_k^\mu)_k]|_{U_\lambda \cap U_\mu}$ $\forall \lambda,\mu$. Let $(\chi_\lambda)_\lambda$ be a partition of unity on $U$ subordinate to $(U_\lambda)_\lambda$.
We define $\vec\varphi_k \in \SK U$ by
\begin{equation}\label{def_skglue}
 \vec\varphi_k(x) \coleq \sum_\lambda \chi_\lambda(x) \cdot \vec\varphi^{\lambda}_k(x)\qquad (x \in U).% \cdot \theta_\lambda
\end{equation}
%Note that $\vec\varphi_k^\lambda \cdot \theta_\lambda \in \cD(U_\lambda)$ extends trivially to $U$. This sum is an element of $\SK U$.
We claim that $(\vec\varphi_k)_k$ is localizing: take $x_0 \in U$ and an open, relatively compact neighborhood $V$ of $x_0$ intersecting only finitely many $\supp \chi_\lambda$, namely those for $\lambda$ in some finite index set $F$. Given $r>0$, for each $\lambda \in F$ choose $k_\lambda$ such that $\forall x \in \supp \chi_\lambda \cap V$ $\forall k \ge k_\lambda$: $\supp \vec\varphi_k^{\lambda}(x)\subseteq \overline{B_r}(x)$ . Then for $k \ge \max_{\lambda \in F} k_\lambda$ and $x \in V$, $\supp \vec\varphi(x) \subseteq \overline{B_r}(x)$ also.

Now we are going to show that $[(\vec\varphi_k)_k]|_{U_\mu} = [(\vec\varphi_k^\mu)_k]$. With $x_0,V,F$ as above, suppose additionally that $x_0 \in U_\mu$ and $V$ is relatively compact in $U_\mu$. By Theorem \ref{sk_restriction} \eqref{9.1} $(\rho^{SK}_{U_\mu, U} \vec\varphi_k)|_V = \vec\varphi_k|_V$ for large $k$; because
$\vec\varphi_k^\lambda(x) = \vec\varphi_k^\mu(x)$ for $x \in \supp \chi_\lambda \cap V$ and large $k$ we have
\[ (\rho^{SK}_{U_\mu, U} \vec\varphi_k )(x) = \sum_{\lambda \in F} \chi_\lambda(x) \cdot \vec\varphi_k^\mu(x) = \vec\varphi_k^\mu(x) \qquad (x \in V) \]
for large $k$, which is what we needed to show.
\end{proof}
The following result about extension of smoothing kernels follows directly from the fact that $\StSKloc$ is a sheaf of $C^\infty$-modules and $C^\infty$ is a fine sheaf.
 \begin{corollary}\label{sk_soft}
Given $V \subseteq U$ open, $W$ open with $\overline{W} \subseteq V$ and $[(\vec\varphi_k)_k]$ in $\SKlocnets V$, there exists $[(\vec\psi_k)_k] \in \SKlocnets U$ such that $[(\vec\psi_k)_k]|_W = [(\vec\varphi_k)_k]|_W$.
\end{corollary}
\begin{proof}
Choose $(\vec\psi_k^0)_k \in \SKloc U$ and $\chi \in C^\infty(U)$ with $\chi \equiv 1$ on $W$ and $\supp \chi \subseteq W'$, where $W'$ is an open set such that $\overline{W} \subseteq W' \subseteq \overline{W'} \subseteq V$.
Then, $[ (\vec\psi_k^0)_k ]|_{U \setminus W'}$ and $ ( (1-\chi) \cdot [(\vec\psi^0_k)_k])|_V + \chi|_V \cdot [(\vec\varphi_k)_k]$ coincide on $V \cap U \setminus W' = V \setminus W'$, hence by Proposition \ref{sk_sheaf} there exists $[(\vec\psi_k)_k] \in \SKlocnets U$ such that $[(\vec\psi_k)_k]|_W = [(\vec\varphi_k)_k]|_W$.
\end{proof}

We now can show that elements of $\SbasEloc$, evaluated on localizing sequences of smoothing kernels, form a sheaf. This is a preliminary step to showing the sheaf property for Colombeau algebras, before any testing is involved. The corresponding space is the following.
\begin{definition}We define 
\[ \cN(U) \coleq \{ R \in \basEloc U\ |\ \forall [(\vec\varphi_k)_k] \in \SKlocnets U: R([(\vec\varphi_k)_k]) = 0\ \textrm{ in }\Cinfnets U\} \]
and set
$\basElocq U \coleq \basEloc U / \cN(U)$. We write $R \sim S$ for $R-S \in \cN(U)$.
\end{definition}

\begin{proposition} Restriction descends to $\basElocq U$ by setting $[R]|_V \coleq [R|_V]$, hence $\SbasElocq$ is a presheaf of vector spaces on $\Omega$.
\end{proposition}

\begin{proof}
We need to show that $R \sim 0$ implies $R|_V \sim 0$. Let $[(\vec\varphi_k)_k] \in \SKlocnets{V}$. Given any $x \in V$, let $W$ be an open neighborhood of $x$ with $\overline{W}$ compact and contained in $V$. Using Corollary \ref{sk_soft} choose $(\vec\psi_k)_k \in  \SKloc U$ such that $\vec\psi_k|_W = \vec\varphi_k|_W$ for large $k$. 
Then $R|_V(\vec\varphi_k)|_W = R(\vec\psi_k)|_W$, which eventually vanishes by assumption. This means $R|_V(\vec\varphi_k)\sim 0$ and hence $R|_V \sim 0$.
\end{proof}

The following shows that elements of $\basElocq U$ are in fact presheaf morphisms.
\begin{lemma}\label{restrpresheaf}
Let $V \subseteq U$ be open, $R \in \basEloc{U}$ and $[(\vec\varphi_k)_k] \in \SKloc{U}$. Then
\[ R([(\vec\varphi_k)_k])|_V = R|_V ( [ (\vec\varphi_k)_k ]|_V ). \]
\end{lemma}
\begin{proof}
Let $x \in V$ and take an open neighborhood $W$ of $x$ such that $\overline{W}$ is compact and contained in $V$. By Theorem \ref{sk_restriction} \eqref{9.1}, $\rho^{SK}_{V,U}(\vec\varphi_k)|_W = \vec\varphi_k|_W$ for large $k$. By definition, $R|_V ( \rho^{SK}_{V,U}(\vec\varphi_k))|_W = R(\vec\varphi_k)|_W$ for large $k$, which is the claim.
\end{proof}

\begin{theorem}
$\SbasElocq$ (as well as $\SbasEploc$ and $\SbasEpi$) is a sheaf of vector spaces on $\Omega$.
\end{theorem}
\begin{proof}
First, let $U \subseteq \Omega$ be open and $(U_\lambda)_\lambda$ an open cover of $U$. Assume $[R] \in \basElocq U$ satisfies $[R]|_{U_\lambda} = 0$ $\forall \lambda$. We claim that $[R]=0$. Let $[(\vec\varphi_k)_k] \in \SKlocnets U$. Any $x \in U$ has an open neighborhood $V$ such that $\overline{V}$ is compact and contained in some $U_\lambda$, hence by Lemma \ref{restrpresheaf} $R ( [ \vec\varphi_k)_k] )|_V = R|_{U_\lambda} ( [ (\vec\varphi_k)_k ] |_{U_\lambda} )|_V$, which implies $[R] = 0$.

Now suppose we are given $[R_\lambda] \in \basElocq{U_\lambda} $ for each $\lambda$ with $[R_\lambda]|_{U_\lambda \cap U_\mu} = [R_\mu]|_{U_\lambda \cap U_\mu}$ $\forall \lambda,\mu$. We need to define $R \in \basEloc{U}$ such that $[R]|_{U_\mu} = [R_\mu]$ $\forall \mu$. Choose a smooth partition of unity $(\chi_\lambda)_\lambda$ on $U$ subordinate to $(U_\lambda)_\lambda$ and
define $R \in \basEloc{U}$ by
\[ R(\vec\varphi) \coleq \sum_\lambda \chi_\lambda \cdot R_\lambda ( \rho^{SK}_{U_\lambda, U} ( \vec\varphi )) \]
For any open set $V$ such that $\overline{V}$ is compact and contained in $U$ there is a finite index set $F$ such that
\[ R(\vec\varphi)|_V = \sum_{\lambda \in F} \chi_\lambda|_V \cdot R_\lambda ( \rho^{SK}_{U_\lambda, U} ( \vec\varphi )|_V. \]
$R$ is smooth if $|_V \circ R \circ c \in C^\infty(\SK U, C^\infty(V))$ for all $c \in C^\infty(\bR, \SK U)$ and such $V$, which is the case because $\rho^{SK}_{U_\lambda, U}$ is smooth for each $\lambda$.

For locality of $R$, suppose $\vec\varphi|_W = \vec\psi|_W$ for $W \subseteq U$ open. Covering $W$ by open subsets $V$ as above, locality follows from Theorem \ref{sk_restriction} \eqref{9.2}, and equally for point-locality and point-independence.

We now need to show that $[R]|_{U_\mu} = [R_\mu]$. We begin by fixing $(\vec\varphi_k)_k \in \SKloc{U_\mu}$ and $x \in U_\mu$, which has an open neighborhood $V$ intersecting only finitely many $\supp \chi_\lambda$, namely those for $\lambda$ in some finite set $F$. We can assume that $\overline{V}$ is compact and contained in $U_\mu \cap \bigcup_{\lambda \in F} U_\lambda$. Choose $(\vec\psi_k)_k \in \SKloc U$ such that $\vec\psi_k|_V = \vec\varphi_k|_V$ for large $k$. Then
\begin{equation}\label{sheafalpha}
 R|_{U_\mu} (\vec\varphi_k)|_V = R (\vec\psi_k)|_V = \sum_{\lambda \in F} \chi_\lambda|_V \cdot R_\lambda ( \rho^{SK}_{U_\lambda, U} ( \vec\psi_k ))|_V.
\end{equation}
For all $\lambda \in F$ then
\begin{align*}
 R_\lambda( [ (\vec\psi_k)_k ]|_{U_\lambda} )|_{U_\lambda \cap U_\mu} & = R_\lambda|_{U_\lambda \cap U_\mu} ( [(\vec\psi_k)_k]|_{U_\lambda \cap U_\mu} ) \\
 = R_\mu|_{U_\lambda \cap U_\mu} ( [(\vec\psi_k)_k]|_{U_\lambda \cap U_\mu} )
 & = R_\mu( [ (\vec\psi_k)_k ]|_{U_\mu} )|_{U_\lambda \cap U_\mu}.
\end{align*}
Because eventually $\rho^{SK}_{U_\mu, U} ( \vec \psi_k)|_V = \vec\psi_k|_V = \vec\varphi_k|_V$ this means that $R_\lambda(\rho^{SK}_{U_\lambda, U} ( \vec\psi_k))|_V = R_\mu ( \rho^{SK}_{U_\mu, U} ( \vec\psi_k ))|_V = R_\mu ( \vec\varphi_k )|_V$  for large $k$
and consequently \eqref{sheafalpha} becomes
\[ R|_{U_\mu} ( \vec\varphi_k)|_V = \left(\sum_{\lambda \in F} \chi_\lambda|_V \right) \cdot R_\mu(\vec\varphi_k)|_V = R_\mu(\vec\varphi_k)|_V \]
which is what we wanted to show.
\end{proof}

What we have shown contains the essence of the proof of the sheaf property for Colombeau algebras. Some adaptions are necessary in order to transfer these results to concrete Colombeau algebras. First of all, the test objects employed satisfy more conditions than localizability and can also be graded, cf.~the example at the end of Section \ref{sec_testing}. Restriction (Theorem \ref{sk_restriction}) and glueing together (Proposition \ref{sk_sheaf}) of such sequences of smoothing kernels need to preserve all these properties. In practice, however, this is easy to verify using Theorem \ref{sk_restriction} \eqref{9.1} because these are only local conditions.

Furthermore, $\basElocq U$ has to be replaced by a quotient $\cG(U) = \cM(U) / \cN(U)$, where a subalgebra $\cM(U)$ of moderate functions containing an ideal $\cN(U)$ of negligible functions are defined via the usual tests. Because these tests also only specify local conditions, one can verify without effort that the results obtained for $\SbasElocq$ also hold for such a quotient $\cG$. Again, this is easily seen for the exemplary algebra given in Section \ref{sec_testing}.

From the above results we can infer that the failure of $\hat\cG^r_s$ to be a sheaf (cf.~\cite[Remark 8.11]{global2}) can be attributed to missing localization properties of its basic space. Finally,
we show some properties of the embedding
\[ \tilde \iota_U\colon  \cD'(U) \to \basEloc U \to \basElocq U \]
which for any open subset $U \subseteq \Omega$ is given by the composition of $\iota \colon \cD'(U) \to \basEloc U$ and the quotient map into $\basElocq U$. These results transfer directly to concrete Colombeau algebras, as well.
\begin{proposition}Let $U \subseteq \Omega$ be open. Then the following holds:
\begin{enumerate}[($i$)]
 \item\label{15.1} $\tilde\iota_U$ is injective.
 \item\label{15.2} $\tilde\iota$ is a sheaf morphism, i.e., $(\tilde \iota_U u)|_V  = \tilde \iota_V ( u|_V)$ for any $V \subseteq U$ open and $u \in \cD'(U)$.
 \item\label{15.3} $\supp u = \supp \tilde \iota_U u$ for $u \in \cD'(U)$.
\end{enumerate}
\end{proposition}
\begin{proof}
 \eqref{15.1}: Assume $\tilde\iota_U u=0$ for $u \in \cD'(U)$. Choose $(\vec\varphi_k)_k \in \SKloc U$ such that the corresponding sequence $(\Phi_k)_k$ in $\cL(\cD'(U), C^\infty(U))$ converges to the identity in $\cL_b(\cD'(U), \cD'(U))$.
Now $\tilde \iota_U u=0$ implies $(\langle u, \vec\varphi_k\rangle)_k \sim 0$, hence any $x \in U$ has an open neighborhood $V$ such that $\langle u, \vec\varphi_k \rangle|_V = 0$ for large $k$; because $\langle u, \vec\varphi_k \rangle|_V \to u|_V$ in $\cD'(V)$ this implies $u|_V = 0$ and hence $u=0$.

\eqref{15.2}: Let $[(\vec\varphi_k)_k] \in \SKlocnets V$ and $W$ open such that $\overline{W}$ is compact and contained in $V$. Using Corollary \ref{sk_soft} choose $[(\vec\psi_k)_k] \in \SKlocnets U$ such that $\vec\psi_k|_W = \vec\varphi_k|_W$ for large $k$. Then
\begin{align*}
 (\iota_U u)|_V ( \vec\varphi_k)|_W &= (\iota_U u)(\vec\psi_k)|_W = \langle u, \vec\psi_k\rangle|_W \\
 &= \langle u|_V, \vec\varphi_k\rangle|_W = \iota_V ( u|_V)(\vec\varphi_k)|_W
\end{align*}
for large $k$, which implies the claim.

\eqref{15.3}: $x \in \supp u$ if and only if for any open neighborhood $V$ of $x$, $u|_V \ne 0$. Because of \eqref{15.1} and \eqref{15.2} $u|_V \ne 0$ is equivalent to $\tilde  \iota_V (u_V) = \tilde \iota_U(u)|_V \ne 0$. Hence, $\supp u = \supp \tilde \iota_U u$.
\end{proof}
Summarizing, we have seen how the sheaf property of Colombeau algebras directly depends on the use of localizing sequences of smoothing kernels. Our presentation allows for the independent study of localization properties of generalized functions at the level of the basic space on the one hand and of properties of sequences of smoothing operators on the other hand, even before any testing for moderateness and negligibility is involved.

\section{Derivatives}\label{sec_deriv}

There are two different approaches to the notion of Lie derivative in $\basE\Omega$. The first one, common to all Colombeau algebras, is of a geometric nature and defines the Lie derivative via the pullback along the flow of a vector field. For a \emph{complete} vector field $X$ on $\Omega$ we hence define the Lie derivative along $X$ of $R \in \basE\Omega$ as
\begin{equation}\label{alpha}
 \LE_X R \coleq \left.\frac{\ud}{\ud t}\right|_{t=0} (\Fl^X_t)^* R \in \basE\Omega
\end{equation}
where $\Fl^X_t$ is the flow of $X$ at time $t$. Because \eqref{alpha} needs $\Fl^X_t$ to be globally defined for small $t$ this definition does not apply in the case $X$ is not complete. Applying the chain rule \cite[Theorem 3.18]{KM} to \eqref{alpha}, however, we obtain an expression for $\LE_XR$ which makes sense for arbitrary $X$:
\begin{definition}
 The \emph{Lie derivative} of $R \in \basE\Omega$ along the vector field $X \in C^\infty(\Omega, \bR^n)$ is defined as
\begin{equation}\label{geomlie}
(\LE_X R)(\vec \varphi) = -\ud R (\vec \varphi) (\Lsk_X\vec\varphi) + \Lsm_X (R(\vec \varphi)) \qquad (\vec\varphi \in \SK\Omega).
\end{equation}
Here, $\Lsk_X\vec\varphi \coleq \Lsm_X \vec\varphi + \Lnf_X \circ \vec \varphi$ is the Lie derivative of $\vec\varphi$ obtained in the same manner by pullback along the flow and the chain rule.
\end{definition}

\begin{remark}\label{rem_derivatives}
Note that in the definition of $\Lsk_X\vec\varphi$, two different Lie derivatives appear: first, $\Lsm_X\vec\varphi = \Lsm_X [ x \mapsto \vec\varphi(x) ]$ is the usual directional derivative of smooth functions. Second, $(\Lnf_X \circ \vec\varphi)(x) = \Lnf_X ( \vec\varphi(x) )$ is the $n$-form derivative. It will always be clear from the context which of the derivatives is being used, which is why we refrained from introducing different symbols for them. However, one has to take care in order to avoid ambiguous expressions like $\Lie_X\vec\varphi(x)$ which could be read as either $(\Lsm_X\vec\varphi)(x)$ or $\Lnf_X ( \vec\varphi(x))$. In the context of $\hat\cG$, these derivatives were denoted $\Lie'_X$ and $\Lie_X$ \cite[Section 3]{global}.
\end{remark}

The second notion of Lie derivative in $\basE\Omega$ comes from the idea of extending operations from $C^\infty(\Omega)$ to $\basE\Omega$ by applying them for fixed $\vec\varphi$:
\begin{definition}For $R \in \basE\Omega$ and $X \in C^\infty(\Omega, \bR^n)$ we set
\[
\tLE_X R \coleq  \Lsm_X \circ R.
\]
\end{definition}
The importance of $\tLE_X$ lies in the fact that it is $C^\infty(\Omega)$-linear in $X$, which will be a crucial property to have in any algebra of generalized tensor fields in order to retain classical tensor calculus. It is conceptually new to have both $\LE_X$ and $\tLE_X$ in the same algebra, as up to now only one of them was available in at a time: $\LE_X$ in full algebras and $\tLE_X$ in special ones. Our basic space $\basE\Omega$ permits to have both at the same time, hence bridges the gap between full and special algebras.

Analogously to the above procedure leading to $\LE_X$ one obtains Lie derivatives for the spaces appearing in Proposition \ref{isos}.
\begin{definition}We define the following Lie derivatives:
\begin{align*}
R \in C^\infty(\cD(\Omega), C^\infty(\Omega)): &\quad (\LE_XR)(\varphi) \coleq - \ud R(\varphi)(\Lnf_X\varphi) + \Lsm_X(R(\varphi)) \\
R \in C^\infty(\cD(\Omega)): &\quad (\LE_X R)(\varphi) \coleq -\ud R(\varphi)(\Lnf_X\varphi) \\
R \in C^\infty(\Omega, \cD'(\Omega)): &\quad (\LE_X R)(x)(\varphi) \coleq \langle -R(x), \Lnf_X \varphi \rangle + \langle (\Lsm_X R)(x), \varphi \rangle \\
R \in \cD'(\Omega): &\quad (\LE_X R)(\varphi) \coleq - \langle R, \Lnf_X \varphi\rangle
\end{align*}
\end{definition}
Of course, for $\cD'(\Omega)$ $\LE_X$ is the usual derivative of distributions. The main properties of all the above derivatives are summarized in the following theorem.

\begin{theorem}\label{thm_diffs}
\begin{enumerate}[($i$)]
 \item\label{13.1} The spaces of diagram \eqref{inclusions} are invariant under $\LE_X$.
 \item\label{13.2} The isomorphisms of Proposition \ref{isos} commute with the Lie derivatives $\LE_X$ on the respective spaces.
\item \label{13.3} $\tLE_X$ maps $\basEpi\Omega$ and $\basEploc\Omega$ into $\basEloc\Omega$, leaves $\basEloc\Omega$ and $\basE\Omega$ invariant and preserves linearity.
  \item \label{13.4} $\LE_X$ commutes with $\iota$ and $\sigma$ and is not $C^\infty(\Omega)$-linear in $X$; $\tLE_X$ only commutes with $\sigma$ but is $C^\infty(\Omega)$-linear in $X$.
 \item \label{13.5} For $R = \iota u \in \basE\Omega$ with $u \in \cD'(\Omega)$ and $(\vec\varphi_k)_k$ a sequence in $\SK\Omega$ such that, if $(\Phi_k)_k$ is the corresponding sequence of smoothing operators, $\Phi_k \to \id_{\cD'(\Omega)}$ in $\cL_b(\cD'(\Omega), \cD'(\Omega))$, we have $(\tLE_XR - \LE_XR)(\vec\varphi_k) \to 0$ in $\cD'(\Omega)$.
\end{enumerate}
\end{theorem}
\begin{proof}
\eqref{13.1}: In the case of \eqref{alpha} this would follow immediately from functoriality of the spaces, but as we use \eqref{geomlie} as our definition we have to give an explicit proof.

First, let $R \in \basEloc\Omega$ and suppose $\vec\varphi|_U = \vec\psi|_U$. Then $R(\vec\varphi)|_U = R(\vec\psi)|_U$ and hence $\Lsm_X(R(\vec\varphi))|_U = \Lsm_X(R(\vec\psi))|_U$. Furthermore, $(\Lsk_X\vec\varphi)|_U = (\Lsm_X \vec\varphi)|_U + \Lnf_X \circ \vec\varphi|_U = (\Lsm_X\vec\psi)|_U + \Lnf_X \circ \vec\psi|_U = (\Lsk_X\vec\psi)|_U$ and thus
$\ud R(\vec\varphi)(\Lsk_X\vec\varphi)|_U = (\frac{\ud}{\ud t}|_{t=0} R(\vec\varphi + t \cdot \Lsk_X\vec\varphi))|_U = \frac{\ud}{\ud t}|_{t=0} ( R(\vec\varphi + t \cdot \Lsk_X\vec\varphi)|_U ) = \frac{\ud}{\ud t}|_{t=0} ( R(\vec\psi + t \cdot \Lsk_X\vec\psi)|_U) = \ud R(\vec\psi)(\Lsk_X\vec\psi)|_U$ because $|_U \colon C^\infty(\Omega) \to C^\infty(U)$ is linear and continuous. In sum we have $(\LE_XR)(\vec\varphi)|_U = (\LE_XR)(\vec\psi)|_U$.

Second, let $R \in \basEploc\Omega$. From \eqref{plocform} we obtain
\begin{align*}
 \Lsm_X(R(\vec\varphi))(x) &= \ud R ( \vec\varphi(x)^\sim ) ((\Lsm_X \vec\varphi)(x)^\sim)(x) + \Lsm_X(R(\vec\varphi(x)^\sim))(x), \\
\ud R(\vec\varphi)(\Lsk_X\vec\varphi)(x) &= (\frac{\ud}{\ud t}|_{t=0} R (\vec\varphi + t \cdot \Lsk_X\vec\varphi))(x) = \frac{\ud}{\ud t}|_{t=0} ( R(\vec\varphi + t \cdot \Lsk_X\vec\varphi)(x)) \\
&= \frac{\ud}{\ud t}|_{t=0} ( R((\vec\varphi(x) + t \cdot \Lsk_X\vec\varphi(x))^\sim)(x))\\
& = \ud R(\vec\varphi(x)^\sim)((\Lsk_X\vec\varphi(x))^\sim)(x)
\end{align*}
and because $\Lsk_X\vec\varphi = \Lsm_X\vec\varphi + \Lnf_X \circ \vec\varphi$ this results in
\begin{equation}\label{blah3}
(\LE_XR)(\vec\varphi)(x) = -\ud R(\vec\varphi(x)^\sim)((\Lnf_X (\vec\varphi(x)))^\sim)(x) + \Lsm_X ( R(\vec\varphi(x)^\sim))(x).
\end{equation}
from which $\LE_XR \in \basEploc\Omega$ follows.

Third, let $R \in \basEpi\Omega$. \eqref{piform} shows that the second term in \eqref{blah3} vanishes and
\begin{equation}\label{blah4}
(\LE_XR)(\vec\varphi)(x) = -\ud R(\vec\varphi(x)^\sim)((\Lnf_X (\vec\varphi(x)))^\sim) \in \bC \subseteq C^\infty(\Omega),
\end{equation}
which obviously defines an element of $\basEpi\Omega$.

Finally, if $R$ is linear then \eqref{geomlie} reduces to $(\LE_XR)(\vec\varphi) = - R(\Lsk_X\vec\varphi) + \Lsm_X(R(\vec\varphi))$ which again is linear in $\vec\varphi$.

\eqref{13.2}: We denote the isomorphism \eqref{alphaa} by $R \mapsto S_R$ with inverse $S \mapsto R_S$ and note that $(\ud R_S)(\vec\varphi)(\vec\psi)(x) = \ud S(\vec\varphi(x))(\vec\psi(x))(x)$.
Then, by \eqref{blah3},
\begin{align*}
(R_{\LE_X S})(\vec\varphi)(x) &= (\LE_X S)(\vec\varphi(x))(x) \\
&= - \ud S(\vec\varphi(x))(\Lnf_X(\vec\varphi(x)))(x) + \Lsm_X(S(\vec\varphi(x)))(x) \\
&= - \ud R_S (\vec \varphi(x)^\sim) ( (\Lnf_X (\vec\varphi(x)))^\sim)(x) + \Lsm_X ( R_S ( \vec\varphi(x)^\sim ) )(x) \\
&= (\LE_X R_S) ( \vec\varphi(x)^\sim )(x) = (\LE_XR_S)(\vec\varphi)(x).
\end{align*}
For $R \in \basEpi\Omega$ we similarly have $(\ud R_S)(\vec\varphi)(\vec\psi)(x) = \ud S (\vec\varphi(x))(\vec\psi(x))$ and by \eqref{blah4}
\begin{align*}
 (R_{\LE_X S})(\vec\varphi)(x) &= (\LE_X S)(\vec\varphi(x))
= - \ud S (\vec\varphi(x))(\Lnf_X(\vec\varphi(x))) \\
&= - \ud R_S (\vec\varphi(x)^\sim)((\Lnf_X(\vec\varphi(x)))^\sim)
= (\LE_X R_S)(\vec \varphi)(x).
\end{align*}
The linear case is clear from this.

\eqref{13.3}: That $\tLE_X$ leaves $\basEloc\Omega$ invariant is obvious from the definitions. For $R \in \basLloc\Omega$ we have $(\tLE_X R)(\vec \varphi)(x) = \langle R, (\Lsm_X \vec\varphi)(x) \rangle$ which clearly is in $\basLloc\Omega \setminus \basLploc\Omega$, hence $\basEpi\Omega$ and $\basEploc\Omega$ are not invariant under $\tLE_X$.

\eqref{13.4} is clear from \eqref{13.2}.

For \eqref{13.5} we regard
\begin{align*}
 (\tLE_X (\iota u) - \LE_X(\iota u))(\vec\varphi_k) & = \langle u, \Lsm_X \vec\varphi_k \rangle + \langle u, \Lnf_X \circ \vec\varphi_k \rangle   \\
& = \langle u, \Lsm_X \vec\varphi_k \rangle - \langle \Lie_X u, \vec\varphi_k \rangle \\
&= \Lsm_X(\Phi_k(u)) - \Phi_k ( \Lie_X u) \to 0\quad\textrm{in }\cD'(\Omega).\qedhere
\end{align*}
\end{proof}

We remark that the above derivatives extend to $\basElocq$ componentwise as $\LE_X [R] \coleq [\LE_X R]$ and $\tLE_X [R] \coleq [\tLE_X R]$. Furthermore, one can also define the Lie derivative $\LE_X$ for arbitrary vector fields locally by \eqref{alpha} and show that it indeed is given by formula \eqref{geomlie}.

From \eqref{13.3} it follows that $\tLE_X$ cannot be defined intrinsically in $\hat\cG$ \cite{global} because it leads out of its basic space; the same reasoning applies in $\hat\cG^r_s$ \cite{global2}, where one would need a derivative which is $C^\infty$-linear in the vector field $X$ in order to define a meaningful covariant derivative. The solution is to use a basic space of \emph{local} functions, where $\tLE_X$ is well-defined. Furthermore, by \eqref{13.5} in any Colombeau algebra constructed on $\SbasEloc$, $\tLE_X (\iota u)$ will be \emph{associated} to $\LE_X (\iota u)$ for all distributions $u$, which is an important property to have in order to esablish compatibility with classical distribution theory.

\section{Simplified Colombeau algebras}\label{sec_realize}

In order to relate our setting to the classical theory we will now show how the simplified Colombeau algebra $\cG^s$ can be obtained from the basic space $\basE\Omega$ by imposing the corresponding tests.

First, we recall the definition of $\cG^s$ in detail. For ease of presentation we will use the variant where representatives of generalized functions are \emph{sequences} of smooth functions \cite{1126.46030}, as opposed to the more commonly used \emph{nets} of smooth functions indexed by $(0,1]$; for the latter case the construction of Theorem \ref{rel_simplified} below would be technically more involved.

In the following, we write $\check\varphi(y) \coleq \varphi(-y)$ and $(\tau_x \varphi)(y) \coleq \varphi(y-x)$ for any $\varphi \in \cD(\Omega)$.

$\cE^s(\Omega) \coleq C^\infty(\Omega)^\bN$ is the \emph{basic space} of the simplified algebra. $\cE_M^s(\Omega) \coleq \{ (u_k)_k \in \cE^s(\Omega)\ |\ \forall p \in \cs(C^\infty(\Omega))\ \exists N \in \bN: p(u_k) = O(k^N)\ (k \to \infty) \}$ is the subset of moderate elements and $\cN^s(\Omega) \coleq \{ (u_k)_k \in \cE^s(\Omega)\ |\ \forall p \in \cs(C^\infty(\Omega))\ \forall m \in \bN: p(u_k) = O(k^{-m})\ (k \to \infty) \}$ the subset of negligible elements. Partial derivatives on $\cE^s(\Omega)$ are defined as $\D^s_i ((u_k)_k) \coleq (\pd_i u_k)_k$ for $i=1,\dotsc,n$.

For the embedding of distributions into $\cG^s(\Omega)$ one employs a mollifier $\rho \in \cS(\bR^n)$ (the Schwartz space of rapidly decreasing functions) satisfying
\[ \int \rho(x)\,\ud x = 1,\qquad \int x^\alpha \rho(x)\,\ud x = 0\quad \forall \alpha>0. \]
We set $\rho_k(x) \coleq k^n \rho(kx)$ for $k \in \bN$.  Then, compactly supported distributions are embedded into $\cE^s(\Omega)$ by convolution via
\[ \iota^s u \coleq ((u * \rho_k)|_\Omega)_k \in \cE^s(\Omega)\qquad (u \in C^\infty(\Omega)'). \]
The fact that $\cG^s(\Omega)$ is a sheaf \cite[Theorem 1.2.4]{GKOS} determines an embedding of $\cD'(\Omega)$ as follows: cover $\Omega$ by a family $(\Omega_\lambda)_\lambda$ of open relatively compact subsets. Choose functions $\psi_\lambda \in \cD(\Omega)$ for each $\lambda$ such that $\psi_\lambda = 1$ on a neighborhood of $\overline{\Omega_\lambda}$ and a partition of unity $(\chi_j)_{j \in \bN}$ subordinate to $\Omega_\lambda$ where each $\chi_j$ has support in $\Omega_{\lambda_j}$ for some $\lambda(j)$. Then the formula for the embedding $\iota^s$ is (see \cite[Equation (1.8)]{GKOS})
\[ \iota^s u \coleq \sum_{j=1}^\infty \chi_j ((\psi_{\lambda(j)} u ) * \rho_k)_k = \langle u, \vec\psi_k \rangle \qquad (u \in \cD'(\Omega)) \]
where $\vec\psi_k \in \SK\Omega$ is given by $\vec\psi_k(x) \coleq \sum_{j} \chi_j(x) \cdot \psi_{\lambda(j)} \cdot \tau_x\check\rho_k$.
We point out that $(\LE_{\pd_i} \vec\psi_k)_k \sim 0$ in $\SKlocnets \Omega$ for all $i$, as is easily verified.

Smooth functions are embedded via
\[ \sigma^s f \coleq (f)_k \qquad (f \in C^\infty(\Omega)). \]
Now $\iota$ and $\sigma$ map into $\cE^s_M(\Omega)$, $(\iota - \sigma)(C^\infty(\Omega)) \subseteq \cN^s(\Omega)$, $\iota(\cD'(\Omega)) \cap \cN^s(\Omega) = \{0\}$, $\cE^s_M(\Omega)$ is a subalgebra of $\cE^s(\Omega)$ and $\cN^s(\Omega)$ an ideal in $\cE^s_M(\Omega)$ and $\D_i$ preserves moderateness and negligibility; $\cG^s(\Omega) \coleq \cE_M(\Omega) / \cN(\Omega)$ then is called the \emph{simplified Colombeau algebra}.

In $\cG^s(\Omega)$, moderateness and negligiblity of embedded objects are determined by testing with $(\vec\psi_k)_k$. We now impose the same tests on $\cE(\Omega)$:

\begin{definition}$R \in \basE\Omega$ is called \emph{$\cG^s$-moderate} if $\forall p \in \cs(C^\infty(\Omega))$ $\exists N \in \bN$: $p(R(\vec\psi_k)) = O(k^N)$ ($k \to \infty$); $R$ is called \emph{$\cG^s$-negligible} if $\forall p \in \cs(C^\infty(\Omega))$ $\forall m \in \bN$: $p(R(\vec\psi_k)) = O(k^{-m})$ ($k \to \infty$). By $\cE_M(\Omega)$ and $\cN(\Omega)$ we denote the subsets of $\cE(\Omega)$ consisting of moderate and negligible elements, respectively.
\end{definition}
The embeddings $\iota$ and $\sigma$ are given by \eqref{iotadef} and \eqref{sigmadef}; partial derivatives are given by $\D_i \coleq \LE_{\pd_i}$. Exactly as in the case of $\cG^s$ (cf.~\cite[Section 1.2]{GKOS}) it follows that $\iota$ and $\sigma$ map into $\cE_M(\Omega)$, $(\iota - \sigma)(C^\infty(\Omega)) \subseteq \cN(\Omega)$, $\iota(\cD'(\Omega)) \cap \cN(\Omega) = \{0\}$, $\cE_M(\Omega)$ is a subalgebra of $\cE(\Omega)$ and $\cE_N(\Omega)$ is an ideal in $\cE_M(\Omega)$. Because $\Lsk_{\pd_i} \vec\psi_k = 0$ the partial derivatives $\D_i$ preserve moderateness and negligibility. We set $\cG(\Omega) \coleq \cE_M(\Omega) / \cN(\Omega)$.

Both $\cG^s(\Omega)$ and $\cG(\Omega)$ are associative commutative algebras with unit containing $\cD'(\Omega)$ injectively as a linear subspace and $C^\infty(\Omega)$ as a subalgebra; furthermore, they are differential algebras whose derivations extend the usual partial derivatives of $\cD'(\Omega)$ (see \cite[Section 1.3]{GKOS} for the general scheme of construction we followed here).

In order to relate $\cG^s(\Omega)$ and $\cG(\Omega)$ to each other we denote by $F\colon \bN \to \SK\Omega$ the mapping $k \mapsto \vec\psi_k$. Then for $R \in \cE(\Omega)$ the sequence $(R(\vec\varphi_k))_k$ is the pullback $F^*R$ of $R$ under $F$.

We have the following main result about the relation of the simplified algebra $\cG^s(\Omega)$ to $\cG(\Omega)$.

\begin{theorem}\label{rel_simplified}
 \begin{enumerate}[($i$)]
  \item $F^*\colon \cE(\Omega) \to \cE^s(\Omega)$ is surjective.
  \item Given $R \in \cE(\Omega)$, $F^*R$ is moderate or negligible if and only if $R$ is $\cG^s$-moderate or $\cG^s$-negligible, respectively.
  \item $F$ induces a linear isomorphism $\cG(\Omega) \cong \cG^s(\Omega)$ which commutes with the respective embeddings and derivatives:
\[
 \xymatrix{
\cG(\Omega) \ar@<0.5ex>[rr] & & \cG^s(\Omega) \ar@<0.5ex>[ll]\\
& \cD'(\Omega) \ar[ul]^\iota \ar[ur]_{\iota^s}
}
\qquad
\xymatrix{
\cG(\Omega) \ar@<0.5ex>[r] & \cG^s(\Omega) \ar@<0.5ex>[l] \\
\cG(\Omega) \ar[u]^{\D_i} \ar@<0.5ex>[r] & \cG^s(\Omega) \ar[u]_{\D^s_i} \ar@<0.5ex>[l] 
}
\]

 \end{enumerate}
\end{theorem}
\begin{proof}
(i) First, we note that there exists $x_0 \in \Omega$ such that $\snorm{\vec\psi_k(x_0)}_{\infty}$ is strictly increasing with $k \to \infty$. Let $(f_k)_k \in \cE^s(\Omega)$ be given. Choose a sequence $(r_k)_k$ of positive real numbers such that both $r_k$ and $r_{k+1}$ are smaller than $(\snorm{\vec\psi_{k+1}(x_0)}_{\infty} - \snorm{\vec\psi_k(x_0)}_\infty )/2$ for each $k \in \bN$. Set $U_k \coleq \{\varphi \in \cD(\bR)\ |\ \norm{\varphi}_{\infty} \le r_k \}$ and choose bump functions $\chi_k \in C^\infty(\cD(\Omega))$ such that $\supp \chi \subseteq U_k$ and $\chi_k(0) = 1$ for $k \in \bN$ \cite[Proposition 16.7]{KM}. Then
\begin{equation}\label{disjoint_supp}
\supp \chi_j(.-\vec\psi_j(x_0)) \cap \supp \chi_k(. - \vec\psi_k(x_0)) = \emptyset\qquad \text{for }j \ne k.
\end{equation} Define $R \in \basE\Omega$ by
\[ R(\vec\varphi) \coleq \sum_k \chi_k ( \vec\varphi(x_0) - \vec\psi_k(x_0)) \cdot f_k. \]
Because of \eqref{disjoint_supp} this sum is well-defined and smooth; it satisfies $R(\vec\psi_k) = f_k$ for all $k$ and thus $F^*R = (f_k)_k$.

(ii) is evident from the definitions. From this it follows that $(F^*)^{-1}$ is well-defined on equivalence classes $[(u_k)_k]$.

For (iii), $F^*$ induces a linear map $\cG(\Omega) \to \cG^s(\Omega)$ which is injective by (ii). Hence, $\cG^s(\Omega) \cong \cG(\Omega)$. We have $F^*(\iota u) = \iota^s(u)$ and $F^*([\sigma u]) = F^*([\iota u)] = [\iota^s u] = [\sigma^s u]$. Because $(\Lsk_{\pd_i} \vec\psi_k)_k \sim 0$, it results that $[F^*(\D_i R)] = [(\pd_i ( R(\vec\psi_k)))_k] = [\D_i^s ( F^*(R))]$. Finally, we see that 	$(F^*)^{-1} ([\iota^s u]) = [\iota u]$ and $(F^*)^{-1} ( D_i^s [(u_k)_k]) = [\D_i ( (F^*)^{-1} ( (u_k)_k ) )]$ holds by construction.
\end{proof}

The above shows that as soon as the tests are fixed one can use their structure for further modifications of the basic space. In fact, the basic space can incorporate the testing procedure to various degrees (cf.~\cite[Section 9]{found} for a discussion about `separating the basic definition from testing'). This is best exemplified by the simplified algebra $\cG^s$, where the sequence of smoothing kernels used for the test is actually incorporated into the embedding $\iota^s$ and is not visible in the definition of moderateness and negligibility anymore. Similarly, the full algebra $\cG^e$ can be obtained from $\SbasEploc$ after a coordinate transformation of the basic space: by the structure of its tests, smoothness in the first slot of elements of the basic space is then not needed for testing anymore, and the basic space can be enlarged accordingly (cf.~\cite{ColElem}). In a similar manner, the diffeomorphism invariant full algebras $\cG^d$ and $\hat\cG$ can be recovered by imposing their respective tests on $\basEploc\Omega$. 

\section{Conclusion}

We have placed the construction of classical Colombeau algebras into a unifying hierarchy based on the idea of taking smoothing operators as a \emph{variable} used in the smoothing process of distributions. This leads to the general basic space $\basE\Omega$ which explains the formal relationship between various basic spaces of classical Colombeau algebras, and, together with the locality properties discussed in Section \ref{sec_subalg}, their sheaf properties and the nature of the derivative $\tLE_X$. The definition of our general basic space is not only very natural and hence applicable to many situations, but it will also serve as a basis for further study of the quotient construction, and marks the beginning of a structure theory for Colombeau algebras. Furthermore, it will provide the grounds for the introduction of a full algebra of nonlinear generalized tensor fields possessing a covariant derivative. 

\paragraph{Acknowledgements.} The author graciously thanks Michael Grosser for helpful discussions. This work was supported by project P23714 of the Austrian Science Fund (FWF).

\end{document}